\documentclass[reqno, 12pt, reqno]{amsart}

\usepackage{amsfonts, amsthm, amsmath, amssymb, color}
\usepackage{comment}
\usepackage{hyperref}
\hypersetup{colorlinks=false}
\usepackage{bbm}

\usepackage[margin=1.5in]{geometry}

\RequirePackage{mathrsfs} \let\mathcal\mathscr
\numberwithin{equation}{section}

\renewcommand{\phi}{\varphi}

\newcommand{\ZZ}{\mathbb{Z}}

\newcommand{\NN}{\mathbb{N}}

\newcommand{\RR}{\mathbb{R}}

\newcommand{\Z}{\mathbb{Z}}

\newcommand{\cF}{\mathcal{F}}

\renewcommand{\leq}{\leqslant}

\renewcommand{\geq}{\geqslant}

\renewcommand{\bar}{\overline}

\newcommand{\x}{\mathbf{x}}
\newcommand{\y}{\mathbf{y}}

\renewcommand{\v}{\mathbf{v}}

\newcommand{\z}{\mathbf{z}}

\renewcommand{\a}{\mathbf{a}}
\renewcommand{\k}{\mathbf{k}}
\renewcommand{\j}{\mathbf{j}}

\newcommand{\bfx}{\mathbf{x}}

\newcommand{\calD}{\mathcal{D}}
\newcommand{\Del}{\Delta}
\newcommand{\N}{\mathbb{N}}

\newcommand{\tet}{\theta}

\newcommand{\ve}{\varepsilon}

\DeclareMathOperator{\supp}{supp}
\DeclareMathOperator{\dist}{dist}

\renewcommand{\t}{\mathbf{t}}

\renewcommand{\hat}{\widehat}

\newcommand{\calM}{\mathcal{M}}

\newtheorem{thm}{Theorem}[section]
\newtheorem{cor}[thm]{Corollary}
\newtheorem{lem}[thm]{Lemma}
\newtheorem{prop}[thm]{Proposition}

\newtheorem{assp}[thm]{Assumption}
\newtheorem{cond}[thm]{Condition}

\theoremstyle{definition}

\newtheorem{rem}[thm]{Remark}

\numberwithin{equation}{section}

\newcommand{\eps}{\varepsilon}

\newtheorem{theorem}{Theorem}
\newtheorem{conjecture}[theorem]{Conjecture}
\newcommand{\del}{{\delta}}

\newcommand{\R}{{\mathbb{R}}}
\newcommand{\calN}{\mathcal{N}}
\newcommand{\calG}{\mathcal{G}}
\newcommand{\calF}{\mathcal{F}}
\newcommand{\vareps}{{\varepsilon}}
\newcommand{\kap}{\kappa}
\newcommand{\bft}{\mathbf{t}}

\newenvironment{blue}{\color{blue}}{}
\specialcomment{comblue}{\begin{blue}}{\end{blue}}

\newenvironment{yellow}{\color{yellow}}{}
\specialcomment{comyellow}{\begin{yellow}}{\end{yellow}}

\newenvironment{red}{\color{red}}{}
\specialcomment{comyellow}{\begin{red}}{\end{red}}

\begin{document}
\date{\today}

\title[Density of rational points near/on compact manifolds]{Density of rational points near/on compact manifolds with certain curvature conditions}

\author{Damaris Schindler}

\address{Mathematisches Institut\\
Georg-August Universi{\"a}t G{\"o}ttingen\\
Bunsenstrasse 3-5\\
37073 G{\"o}ttingen\\
Germany}
\email{damaris.schindler@mathematik.uni-goettingen.de}

\author{Shuntaro Yamagishi}
\address{Mathematisch Instituut\\
Universiteit Utrecht\\
Budapestlaan 6\\
NL-3584CD Utrecht\\
The Netherlands}
\email{s.yamagishi@uu.nl}


\begin{abstract}
In this article we establish an asymptotic formula for the number of rational points, with bounded denominators, within a given distance to a compact submanifold $\calM$ of $\R^M$ with a certain curvature condition. Our result generalises earlier work of Huang for hypersurfaces
[J.-J. Huang, The density of rational points near hypersurfaces, Duke Math. J. 169 (2020), 2045--2077.],
as our curvature condition reduces to Gaussian curvature being bounded away from $0$ when $M - \dim \calM = 1$. An interesting feature of our result is that the asymptotic formula holds beyond the conjectured range of the distance to $\calM$. Furthermore, we obtain an upper bound for the number of rational points on $\calM$ with additional power saving to the bound in the analogue of Serre's dimension growth conjecture for compact submanifolds of $\RR^M$ when $M - \dim \calM > 1$. 
\end{abstract}

\maketitle

\thispagestyle{empty}
\setcounter{tocdepth}{1}
\tableofcontents

\section{Introduction}
In this article we study the density of rational points `close' to smooth manifolds. Let $\calM$ be a compact (immersed) submanifold of $\RR^M$ and let
$R = M - \dim \calM$ be the codimension of $\calM$.
Given $Q \in \NN$ and $\delta \geq 0$, we let
$$N(\calM; Q, \del):=\# \{  (\a, q) \in \Z^M \times \NN: 1 \leq q \leq Q,  \dist (\a/q, \calM) \leq \del/q\},$$
where
$\dist(\cdot, \cdot)$ denotes the $L^{\infty}$-distance on $\RR^M$. The study of $N(\calM; Q, \del)$ is an active area of research which has seen significant progress recently. It is an interesting problem in its own right, but it also relates to other areas of mathematics such as
Diophantine approximation. There is also a connection to an interesting question raised by Mazur in \cite[pp.331]{BM}
``given  a  smooth curve  in  the  plane,  how  near  to it can a point with rational coordinates get and still miss?''

A trivial estimate for $N(\calM; Q, \del)$ is given by
$$
N(\calM; Q, \del) \ll Q^{\dim \calM + 1},
$$
while a probabilistic heuristic suggests
$$
\delta^R Q^{\dim \calM + 1} \ll N(\calM; Q, \del) \ll \delta^R Q^{\dim \calM + 1}.
$$
We know that this heuristic estimate does not hold in complete generality. For example, if $\calM$ is a rational hyperplane in $\RR^M$ and $\del \leq 1$, then we have
$$
Q^{\dim \calM + 1} \ll N(\calM; Q, \del) \ll Q^{\dim \calM + 1}.
$$
In one of the spectacular achievements in the field \cite{Ber12}, Beresnevich established the following sharp lower bound
\begin{eqnarray}
\label{Bbdd}
N(\calM; Q, \del) \gg_{\calM} \delta^R Q^{\dim \calM + 1} \quad \textnormal{for any } \delta \gg Q^{-\frac{1}{R}},
\end{eqnarray}
assuming $\calM$ is an analytic submanifold of $\RR^M$ which contains at least one nondegenerate (see \cite{Ber12} for the definition) point.
In his groundbreaking work \cite{JJH}, Huang proposed the following conjecture.
\begin{conjecture}\cite[Conjecture 1]{JJH}
\label{conjJJH}
Let $\calM$ be a bounded immersed submanifold of $\R^M$ with boundary. Let $R = M - \dim \calM$. Suppose $\calM$ satisfies `proper' curvature conditions. Then there exists a constant $c_{\calM} > 0$ depending only on $\calM$ such that\footnote
{The statement
$N(\calM; Q, \delta) \sim  c_{\calM} \delta^{R} Q^{\dim \calM +1}$
means $\lim_{Q \rightarrow \infty}
\frac{N(\calM; Q, \delta)
}{c_{\calM} \delta^{R} Q^{\dim \calM +1}} = 1$.}
$$
N(\calM; Q, \delta) \sim  c_{\calM} \delta^{R} Q^{\dim \calM +1}
$$
when $\delta \geq Q^{- \frac{1}{R} + \epsilon}$ for some $\epsilon > 0$ and $Q \rightarrow \infty.$

\end{conjecture}
It is not formulated precisely in \cite{JJH} what `proper' curvature conditions mean in this context.
In the same article Huang established Conjecture \ref{conjJJH} for the case when $\calM$ is a hypersurface with Gaussian curvature bounded away from zero.
Previously, the conjecture was only known to hold for planar curves \cite{HuangC}.
For earlier results towards Conjecture \ref{conjJJH} see for example \cite{BDV07}, \cite{Huxley}, \cite{VV06} for the case of planar curves, and \cite{Ber12}, \cite{BVVZ} for the more general case.
We refer the reader to \cite{JJH} for a more detailed development of the field. \par

In light of (\ref{Bbdd}) and Conjecture \ref{conjJJH}, it is natural to ask whether one can
obtain estimates for $N(\calM; Q, \delta)$ beyond the range $\delta \geq Q^{- \frac{1}{R} + \epsilon}$
under certain curvature conditions. Investigating this problem,
along with constructing a class of manifolds of arbitrary codimension that satisfies Conjecture \ref{conjJJH},
is the primary goal of this article.

We now state our main result in detail. Let $\calM$ be as in Conjecture \ref{conjJJH}. We will work with $\calM$ locally. Thus, in view of the implicit function theorem, we may assume without loss of generality that
\begin{eqnarray}
\label{defM}
\calM : = \{ (\x, f_1(\bfx), \ldots, f_R(\bfx)) \in \RR^M: \x = (x_1, \ldots, x_n) \in \overline{B_{\varepsilon_0}(\x_0)} \},
\end{eqnarray}
where $\x_0 \in \RR^n$, $\varepsilon_0 > 0$ and $f_r\in C^{\ell} (\RR^n)$ $(1\leq r\leq R)$ for some $\ell \geq 2$.
In particular, $\dim \calM = n$.
We note that for a general compact submanifold $\calM \subseteq \RR^M$, estimating $N(\calM; Q, \delta)$
reduces to that for a finite number of subsets of $\calM$ of the form $(\ref{defM})$.

Let $w \in C_0^{\infty}(\mathbb{R}^n)$ be a non-negative weight function with $\supp w$
contained in a sufficiently small (with respect to $\mathcal{M}$) open neighbourhood of
$\x_0$. For $Q\in \N$ and $0 \leq \del \leq \frac12$, we define
\begin{eqnarray}
\label{eqn defn1'}
\mathcal{N}_w(Q, \delta) = \sum_{\substack{ \a \in \mathbb{Z}^n \\ q \leq Q \\ \| q f_1(\a/q) \| \leq \delta  \\ \vdots  \\  \| q f_R(\a/q) \| \leq \delta }}
w \left( \frac{\a}{q} \right),
\end{eqnarray}
where $\| \cdot \|$ denotes the distance to the nearest integer. Since $\| x \| \leq 1/2$ $(x \in \mathbb{R})$, we only consider $0 \leq \delta \leq 1/ 2$.
Let
$$N_0:= \sum_{\substack{ \a \in \mathbb{Z}^n \\ q \leq Q } } w \left( \frac{\a}{q} \right).$$
Given any $f \in C^2(\mathbb{R}^n)$ we denote by $H_f(\x)$ the Hessian matrix of $f$ evaluated at $\x \in \mathbb{R}^n$, i.e.
the $n \times n$ matrix whose $(i,j)$-th entry is $\frac{\partial^2 f}{\partial x_i \partial x_j} (\x)$ $(1 \leq i, j \leq n)$.
The following is the curvature condition we require in this article.
\begin{cond}
\label{assp1}
Given any $(t_1, \ldots, t_R) \in \RR^R \backslash \{ \mathbf{0}\}$,
$$
\det H_{t_1 f_1 + \cdots + t_R f_R} (\mathbf{x}_0) \neq 0.
$$
\end{cond}
We remark that when $R=1$, Condition \ref{assp1} reduces to $\det H_{f_1} (\mathbf{x}_0) \neq 0$.
With these notations we have the following result, which contains \cite[Theorem 2]{JJH}
as a special case when $R = 1$.
\begin{thm}
\label{main upper bound thm}
Let $n\geq 2$ and $\ell > \max\{  n+1, \frac{n}{2}+4\}$. Suppose Condition \ref{assp1} holds and that $\varepsilon_0 > 0$ is sufficiently small.
Then we have
$$
\left| \mathcal{N}_w(Q, \delta) - (2\del)^R N_0\right| \ll
\begin{cases}
\del^{\frac{(R-1)(n-2)}{n}} Q^n  \mathcal{E}_n(Q) & \mbox{ if } 
\delta \geq Q^{- \frac{n}{n+2(R-1)}}
\\
Q^{n -\frac{(n-2) (R-1) }{n + 2(R-1)} }   \mathcal{E}_n(Q)
& \mbox{ if }     
\delta < Q^{- \frac{n}{n+2(R-1)}},
\end{cases}
$$
where
$$
\mathcal{E}_n(Q)
=
\begin{cases}
    \exp( \mathfrak{c}_1 \sqrt{\log Q}) & \mbox{if } n = 2 \\
    (\log Q)^{\mathfrak{c}_2} & \mbox{if } n \geq 3. \\
\end{cases}
$$
for some positive constants $\mathfrak{c}_1$ and $\mathfrak{c}_2$.
Here the constants $\mathfrak{c}_1$ and $\mathfrak{c}_2$ and the implicit constants
depend only on $\mathcal{M}$ and $w$.
\end{thm}
We defer making a technical remark (Remark \ref{REM1}) regarding this theorem to the end of Section \ref{auxinequ}.

By the Poisson summation formula we have that $N_0 = \sigma Q^{n+1} + O(Q^n)$ for some positive constant $\sigma$ depending only on $w$ and $n$ (see \cite[(6.2)]{JJH}).
Therefore, it follows immediately from Theorem \ref{main upper bound thm} that
\begin{eqnarray}
\label{asym1}
\mathcal{N}_w(Q, \delta) = (2\del)^R \sigma Q^{n+1} + O \left(\del^{\frac{(R-1)(n-2)}{n}} Q^{n} \mathcal{E}_n(Q) \right)
\end{eqnarray}
holds when $\delta \geq  Q^{- \frac{n}{n + 2(R-1) } + \epsilon }$ for any $\epsilon > 0$ sufficiently small.
Then by approximating the characteristic function of the set $\overline{B_{\varepsilon_0}(\x_0)}$ by smooth weight functions as in \cite[Section 7]{JJH}, we establish Conjecture \ref{conjJJH} via (\ref{asym1}).
In fact, since
$$
Q^{\frac{-1}{R}} \geq Q^{- \frac{n}{n + 2(R-1) }  }
$$
we obtain the asymptotic formula for $N(\calM; Q, \delta)$ beyond the range of $\delta$ in the conjecture.
\begin{cor} Let $\calM$ be as in (\ref{defM}), $n\geq 2$ and $\ell > \max\{  n+1, \frac{n}{2}+4\}$. Suppose Condition \ref{assp1} holds
and that $\varepsilon_0 > 0$ is sufficiently small. Then there exists a constant $c_{\calM} > 0$ depending only on $\calM$ such that
$$
N(\calM; Q, \delta) \sim  c_{\calM} \delta^{R} Q^{\dim \calM +1}
$$
when $\delta \geq Q^{ - \frac{n}{n + 2(R-1) } + \epsilon}$ for any $\epsilon > 0$ sufficiently small and $Q \rightarrow \infty.$
In particular, Conjecture \ref{conjJJH} holds in this case.
\end{cor}
Although Condition \ref{assp1} reduces to Gaussian curvature being bounded away from $0$ when $R=1$,
for larger values of $R$ it makes sense to first question whether such functions satisfying Condition \ref{assp1}
even exist.
As we shall see in Section \ref{examples}, we can construct functions $f_r$ $(1 \leq r \leq R)$ satisfying Condition \ref{assp1} with $n = 2^{R-1}$ for every $R \geq 2$. Thus, by taking larger values of $R$, we establish the existence of a class of manifolds where the asymptotic in Conjecture \ref{conjJJH} holds with $\delta$ arbitrarily close to $Q^{-1}$.

Furthermore, work of Huang \cite{JJH} played an important role in
the first author's work \cite{DS} generalising Bourgain's result \cite{Bourg} on Diophantine inequalities involving
quadratic ternary diagonal forms to higher degrees.
Thus it is plausible that
Theorem \ref{main upper bound thm} will be found useful in similar, and other, applications as well.

If we consider the case $\del=0$, then $\calN_w(Q,0)$ counts the (weighted) number of rational points {\em on} the manifold $\calM$ with bounded denominators. As explained in \cite[pp. 2047]{JJH} Conjecture \ref{conjJJH} implies 
\begin{equation}
\label{eqnn}
N(\calM; Q, 0)\ll Q^{\dim \calM + \eps}
\end{equation}
for any $\eps > 0$ sufficiently small.
One may view this as an analogue of Serre's dimension growth conjecture, which is stated below, for smooth submanifolds of $\RR^M$.
In general, the upper bound (\ref{eqnn}) is sharp.
\begin{conjecture}[The dimension growth conjecture]\label{Serre}
Let $X \subseteq \mathbb{P}_{\mathbb{Q}}^{M-1}$ be an irreducible projective variety of degree at least two defined over $\mathbb{Q}$. Let $N_X(B)$ be the number of rational points on $X$ of naive height bounded by $B$. Then
$$N_X(B)\ll B^{\dim X} (\log B)^c$$
for some constant $c > 0$.
\end{conjecture}
There is a large body of work regarding the dimension growth conjecture, which (a version with $B^{\epsilon}$ in place of $(\log B)^c$)
is now a theorem due to Salberger \cite{Sal3}. We refer the reader to \cite{CCD} for an introduction to the topic and also to other work related to this conjecture, for example \cite{Broberg}, \cite{BroSal}, \cite{BHB1}, \cite{BHB2}, \cite{BHB3}, \cite{BHBS}, \cite{CCD}, \cite{HB02}, \cite{Marmon}, \cite{Sal1}, \cite{Sal2}, 
\cite{Sal4}, \cite{Walsh}. In general, the upper bound in Conjecture \ref{Serre} is sharp. For example, if $X$ contains a rational linear divisor, then this subvariety already contains $B^{\dim X}$ points of naive height bounded by $B$. However, it is possible to obtain a better upper bound by excluding divisors of small degree and imposing stronger conditions, e.g. on the degree of the variety. For hypersurfaces of degree at least four such a result has been established by Marmon \cite{Marmon}.
In the setting of smooth submanifolds of $\RR^M$, as an immediate consequence of Theorem \ref{main upper bound thm}
we obtain the following estimate breaking the $Q^{\dim \calM}$ barrier in (\ref{eqnn}).
\begin{cor}
Let $\calM$ be as in (\ref{defM}), $n\geq 3$ and $\ell > \max\{  n+1, \frac{n}{2}+4\}$.
Suppose Condition \ref{assp1} holds and that $\varepsilon_0 > 0$ is sufficiently small.
Then
$$
N(\calM; Q, 0) \ll Q^{n -\frac{(n-2) (R-1) }{n + 2(R-1)} } (\log Q)^{c}$$
for some constant $c > 0$.
\end{cor}

We prove Theorem \ref{main upper bound thm} by a combination of the method developed by Huang in \cite{JJH} and fibration arguments.
More precisely, Huang develops a procedure to relate the counting problem for a function to that for its Legendre transform,
and this allows him to iteratively improve the upper bound where about $\log \log Q$ iterations yield the optimal bound. In our situation we show that a similar, but more complicated, procedure can be made to work for a `nice' family (this essentially means Condition \ref{assp1} holds) of functions of the form $f_1 + t_2 f_2 + \cdots + t_R f_R$; however, we apply this procedure twice instead of once as in \cite{JJH}, at which point we are able to reduce the problem to that for one function where we can invoke the main result of \cite{JJH}.
After providing some auxiliary lemmata in Section \ref{prelim}, we prove Theorem \ref{main upper bound thm} in Section \ref{beginproof}--\ref{auxinequ}. In Section \ref{examples} we construct examples of real symmetric $n \times n$ matrices $A_1,\ldots, A_R$ with the property that
\begin{equation}\label{eqnAs}
\det (t_1A_1+\ldots +t_RA_R)\neq 0
\end{equation}
for all $(t_1,\ldots, t_R)\in \R^R \setminus\{\mathbf{0}\}$.
We can easily come up with functions $f_r \in C^{\ell}(\RR^n)$ with $H_{f_r}(\x_0)=A_r$ $(1\leq r\leq R)$, which in particular satisfy
Condition \ref{assp1}. For example, let
$$
f_r(\x) = \frac{1}{2}\sum_{1 \leq i, j \leq n} A_{r; i, j} x_i x_j +  g_r(\x)  \quad (1 \leq r \leq R),
$$
where $A_{r; i, j}$ is the $(i,j)$-th entry of $A_r$ and $g_r \in C^{\ell}(\RR^n)$
is any function such that $H_{g_r}(\x_0)= [0]$ $(1\leq r \leq R)$.

As pointed out to us by Mike Roth,
finding $n\times n$ (not necessarily symmetric) matrices $A_1,\ldots, A_R$ satisfying (\ref{eqnAs}) is closely related to a
well-known problem in homotopy theory and in the theory of fibre bundles concerning the number of linearly independent vector fields on spheres (We refer the reader to \cite{Adams} for more information on this topic.). In fact, such an $R$-tuple of symmetric matrices
gives rise to a system of $(R-1)$ linearly independent vector fields on the $(n-1)$-sphere $S^{n-1}\subseteq \R^n$. In particular, it follows by work of Adams \cite{Adams} that $R\leq \varrho(n)$ where $\varrho(n)$ is the Radon-Hurwitz function.

\vspace{0.2cm}
{\em Acknowledgements.} The authors would like to thank Wilberd van der Kallen and Alejandro Gonz\'{a}lez Nevado for providing 
Examples 1 and   2  respectively in Section \ref{examples}. We thank Jing-Jing Huang, Igor Klep and Markus Schweighofer for helpful discussions and Mike Roth for pointing out to us the connection between matrices satisfying (\ref{eqnAs}) and linearly independent vector fields on spheres. While working on this article the second author was supported by the NWO Veni Grant 016.Veni.192.047.

\section{Notation}
In this article $| \cdot |$ denotes the $L^{\infty}$-norm, i.e. $|\mathbf{z}| = \max_{1 \leq i \leq n} |z_i|$ where
$\z=(z_1,\ldots, z_n)\in \mathbb{R}^n$.  Given an open set $\calD \subseteq \mathbb{R}^n$ we denote by $C^{\ell}(\calD)$ the set of $\ell$-times continuously differentiable functions defined on $\calD$, $C^{\infty}(\calD)$ the set of smooth functions defined on $\calD$, and
$C_0^{\infty}(\calD)$ the set of smooth functions defined on $\calD$ with compact support (i.e. given $g \in C_0^{\infty}(\calD)$ the closure of its support $\supp g = \{ \x \in \calD: g(\x) \neq 0 \}$ is compact).
Given any $f \in C^1(\mathbb{R}^n)$ we let $\nabla f = ( \frac{\partial f}{\partial x_1}, \ldots, \frac{\partial f}{\partial x_n})$.
Given $\ve > 0$ and $\x_0 = (x_{0,1}, \ldots, x_{0,n}) \in \mathbb{R}^n$ we let
$$B_{\ve}(\x_0) = (x_{0,1} - \ve, x_{0,1} + \ve) \times \cdots \times (x_{0,n} - \ve, x_{0,n} + \ve).$$
Given $\mathcal{X} \subseteq \mathbb{R}^n$ we denote the boundary of $\mathcal{X}$ by
$\partial \mathcal{X} = \overline{\mathcal{X}} \backslash \mathcal{X}^{\circ}$,
where $\overline{\mathcal{X}}$ is the closure of $\mathcal{X}$ and $\mathcal{X}^{\circ}$ is the interior of $\mathcal{X}$.
For any $z \in \RR$ we let $e(z) = e^{2 \pi i z}$.
By the notation $f (\x)\ll g(\x)$ or $f(\x) = O(g(\x))$ we mean that there exits a constant $C > 0$ such that $|f (\x)| < C g(\x)$ for all $\x$ in
consideration. 

\section{Preliminaries}\label{prelim}
In this section we collect few facts about Legendre transform and oscillatory integrals, and some compactness results that we use later.\par
Given $F \in C^{\ell}(\mathbb{R}^n)$ let $\mathcal{U} \subseteq \RR^n$ be an open set such that
$\nabla F$ is invertible on $\mathcal{U}$.
We define the Legendre transform $F^*: \nabla F (\mathcal{U}) \to \RR$ of $F$ by
$$
F^*(\y) = \y \cdot (\nabla F)^{-1}(\y) - (F \circ (\nabla F)^{-1})(\y).
$$
It can be verified that $F^*$ is $\ell $-times continuously differentiable, $F^{**} = F$
and $\nabla F^* = (\nabla F)^{-1}$.
Also when $\y = \nabla F(\x)$ we have
\begin{eqnarray}\label{dual}
F^*(\y) = \x \cdot \y - F(\x)
\end{eqnarray}
and
\begin{eqnarray}
\label{Hssnreln}
H_{F^*}(\y) = H_F (\x)^{-1}.
\end{eqnarray}

For the following results we refer the reader to see for example \cite[Theorem 7.7.1 and Theorem 7.7.5]{Hormander}.

\begin{lem}(Non-stationary phase)
\label{Lem1}
Let $\ell \in \NN$ and $U_+ \subseteq \RR^d$ a bounded open set.
Let $\omega \in C_0^{\ell -1}(\RR^d)$ with $\overline{\supp \omega} \subseteq U_+$
and $\phi \in C^{\ell}(U_+)$ with $\nabla \phi(\x) \not = \mathbf{0}$ for all $\x \in \overline{\supp \omega}$.
Then for any $\lambda > 0$
$$
\left|  \int_{\RR^d} \omega(\x)e(\lambda \phi(\x)) d\x \right| \leq c_{\ell} \lambda^{-\ell + 1}.
$$
Furthermore, $c_{\ell}$ depends only on $\ell$, $d$, upper bounds for (the absolute values of) finitely many derivatives\footnote{This includes the zeroth partial derivative, i.e. the function itself, as well.}
of $\omega$ and $\phi$ on $U_+$, and a lower bound for $|\nabla \phi|$ on $\overline{\supp \omega}$.
\end{lem}

Recall given a symmetric matrix we define its signature to be the number of positive eigenvalues minus the number of negative eigenvalues.
\begin{lem}(Stationary phase)
\label{Lem2}
Let $\ell >  \frac{d}{2} + 4$ and $\mathcal{D}, \mathcal{D}_+ \subseteq \RR^d$ bounded open sets such that $\overline{\mathcal{D}} \subseteq \mathcal{D}_+$.
Let $\omega \in C_0^{\ell -1}(\RR^d)$ with $\overline{\supp \omega} \subseteq \mathcal{D}$
and $\phi \in C^{\ell}(\mathcal{D}_+)$. 
Suppose $\nabla \phi(\v_0) = \mathbf{0}$ and $\det H_{\phi}(\v_0) \not = 0$ for some $\v_0 \in \mathcal{D}$. Let
$\sigma$ be the signature of $H_{\phi}(\v_0)$ and $\Delta = |\det H_{\phi}(\v_0)|$.
Suppose further that $\nabla \phi(\x) \not = \mathbf{0}$ for all $\x \in \overline{\mathcal{D}} \setminus \{ \v_0 \}$.
Then
$$
\int_{\RR^d} \omega(\x) e(\lambda \phi(\x)) d \x
= e\left( \lambda \phi(\v_0) + \frac{\sigma}{8} \right) \Delta^{- \frac{1}{2} } \lambda^{- \frac{d}{2}} (u(\v_0) + O(\lambda^{-1})),
$$
where the implicit constant depends only on $\ell$, $d$, upper bounds for (the absolute values of) finitely many derivatives of $\omega$ and $\phi$ on $\mathcal{D}_+$, an upper bound for $|\x - \v_0|/ |\nabla \phi (\x)| $ on $\mathcal{D}_+$, and a lower bound for $\Del$.
\end{lem}
We remark that this is a simplified version of \cite[Theorem 7.7.5]{Hormander}.
The assumption on $\ell$ in Lemma \ref{Lem2} can be deduced from  \cite[pp. 222, Remark]{Hormander}.

\subsection{Compactness results}\label{secIS}
Let $m\in \Z_{\geq 0}$ and $\calG=\calG_1\times \calG_2\subseteq \R^{n+m}$, where $\calG_1\subseteq \R^n$ and $\calG_2\subseteq \R^m$ are bounded connected open sets. Let $G \in C^{\ell}(\calG)$ for some $\ell \geq 2$.
For a fixed vector $\t \in \calG_2$ we write $G_\t:\calG_1\rightarrow \R$ for the function $\bfx \mapsto G(\bfx,\t)$.\par

Let $\bfx_0\in \calG_1$ be a fixed point.

\begin{assp}\label{assp2}
We have $\det H_{G_\t}(\bfx_0)\neq 0$ for every $\t\in \calG_2$.
\end{assp}

\begin{lem}\label{ct1}
Let $G \in C^{\ell}(\calG)$, $\ell \geq 2$, and assume that $G$ satisfies Assumption \ref{assp2}. Let $\calF_2$ be a compact set contained in $\calG_2$. Then there exist a real number $\tau >0$  and constants $c_{\calF_2, 1}, c_{\calF_2, 2} > 0$
such that
$$c_{\calF_2, 1}\leq |\det H_{G_\t}(\bfx)|\leq c_{\calF_2, 2}$$
for all $\bfx \in B_\tau (\bfx_0)$ and $\t\in \calF_2$. Moreover, the map $\bfx\mapsto \nabla G_\t (\bfx)$ is a $C^{\ell-1}$-diffeomorphism on $\bar{B_{\tau}(\bfx_0)}$ for all $\t\in \calF_2$.
\end{lem}

\begin{proof}
We define a function $\psi: \mathbb{R}^{n +m} \rightarrow \mathbb{R}^{n + m}$ by
\begin{eqnarray}
\psi (\x,\t)= ( \nabla G_\t (\x),\t)=
\left( \frac{\partial G_\t}{\partial x_1}(\mathbf{x}) , \ldots,
\frac{\partial G_{\t}}{\partial x_n}(\mathbf{x}) , \t\right).
\end{eqnarray}
Clearly the determinant of the Jacobian matrix of $\psi$ at $(\x, \t)$ is $\det H_{G_{\t}} (\x)$.
Let $\t \in \calG_2$. Then by the inverse function theorem we can find $\vareps_\t>0$ such that
$|\det H_{G_{\t}} (\x)| > 0$ $( (\x, \t) \in  B_{\vareps_\t}(\bfx_0) \times B_{\vareps_\t}(\t) )$
and $\psi$ is a $C^{\ell-1}$-diffeomorphism on $B_{\vareps_\t}(\bfx_0) \times B_{\vareps_\t}(\t)$.
Since $\calF_2$ is compact we can find points $\t_1,\ldots, \t_v\in \calF_2$ such that
$$\calF_2\subseteq \bigcup_{1\leq i\leq v} B_{\vareps_{\t_i}}(\t_i).$$
Let
$$0<\tau < \min_{1\leq i\leq v} \frac{\vareps_{\t_i}}{2}.$$
It is easy to see that the first part of the lemma follows with this choice of $\tau$.
Given any $\t \in \calF_2$ there exists $1\leq i\leq v$ such that $\t \in B_{\vareps_{\t_i}}(\t_i)$ and $\psi$ is a $C^{\ell-1}$-diffeomorphism on $B_{\vareps_{\t_i}}(\bfx_0)\times B_{\vareps_{\t_i}}(\t_i)$. It follows by the shape of the map $\psi$ that $\psi (\cdot, \t)$ is a $C^{\ell-1}$-diffeomorphism on $\bar{B_\tau(\bfx_0)}$.
\end{proof}

\begin{lem}\label{ct2}
Let $G \in C^{\ell}(\calG)$, $\ell \geq 2$, and assume that $G$ satisfies Assumption \ref{assp2}.
Let  $\calF_2$ and $\tau$ be as in Lemma \ref{ct1}. 
Then for any $0<\kap<\tau$ sufficiently small, there exists $\rho>0$ such that
$$\dist \left( \partial \left(\nabla G_\t ( B_\tau(\bfx_0))\right), \partial \left( \nabla G_\t (B_\kap (\bfx_0))\right)\right) \geq 2\rho$$
for all $\t\in \calF_2$.
\end{lem}

\begin{proof}
Let $\t\in \calF_2$ and $0< \vareps<\tau$. Since $\nabla G_\t$ is a diffeomorphism on $\overline{B_\tau(\bfx_0)}$, we have $\partial ( \nabla G_\t (B_\tau(\bfx_0))) = \nabla G_\t (\partial B_\tau(\bfx_0))$
and $\partial (  \nabla G_\t( B_z (\x_0) )) =  \nabla G_\t (\partial B_z (\x_0) )$ for any $z \in [0, \varepsilon)$.
Let
$$
L(z) = \min_{\t\in \calF_2 } \textnormal{dist} \left( \nabla G_\t (\partial B_\tau(\bfx_0)), \nabla G_\t (\partial B_z (\x_0) ) \right),
$$
where $\partial B_{0} (\x_0) = \{\x_0\}$.

First we justify that $L: [0, \ve) \rightarrow [0, \infty)$ is a well-defined function. For each fixed $z \in [0, \ve)$ we have that
\begin{eqnarray}
\notag
\textnormal{dist} \left( \nabla G_\t (\partial B_\tau(\bfx_0)), \nabla G_\t (\partial B_z (\x_0) ) \right)
= \min_{ \substack{ \x \in \partial B_\tau(\bfx_0) \\  \x' \in \partial  B_z (\x_0) } }   |  \nabla G_\t (\x)  -  \nabla G_\t (\x')    |
\end{eqnarray}
is a continuous function in $\t$, because $| \nabla G_\t (\x)  -  \nabla G_\t (\x')   |$ is a continuous function
in $(\x, \x', \t)$ and $\partial B_\tau(\bfx_0) \times \partial B_z (\x_0)$ is compact.
Therefore, since $\calF_2$ is compact it follows that $\min_{\t\in \calF_2 }$ exists in the definition of $L(z)$.\par
It is clear that $L(0) > 0$.
If there exists $z \in  (0, \ve) $ such that $L(z) > 0$ then we are done.
Thus let us suppose otherwise, i.e. $L(z) = 0$ for all $z \in (0, \ve)$. Then there exist $\t \in \calF_2$ and $z \in (0, \ve)$ such that
$$
\nabla G_\t (\partial B_\tau(\bfx_0)) \cap \nabla G_\t (\partial B_z (\x_0) ) \not = \emptyset,
$$
but this is not possible because
$\overline{  B_z (\x_0)  } \subseteq B_\tau(\bfx_0)$  
and $\nabla G_\t$
is a diffeomorphism on $\overline{B_\tau(\bfx_0)}$.
\end{proof}

\section{Setting up the proof of Theorem \ref{main upper bound thm}}\label{beginproof}
Given $0 < \delta \leq 1/2$ we let
\begin{eqnarray}
\label{defchi}
\chi_{\delta}(\theta) 
=
\left\{
    \begin{array}{ll}
         1
         &\mbox{  if  }\| \theta \| \leq \delta \\
         0
         &\mbox{  otherwise.}
    \end{array}
\right.
\end{eqnarray}
We consider the Selberg magic functions (see \cite{Montgomery}) for the interval $[-\del,\del]\subseteq \R/\Z$ and denote them by
$$S_J^{\pm}(x)=\sum_{|j|\leq J} \hat{S}_J^{\pm}(j) e(jx),$$
where $J \in \NN$ is a parameter to be chosen in due course. They have the properties that
$$S_J^-(\tet)\leq \chi_{\delta}(\tet) \leq S_J^+(\tet) \quad (\tet\in \R/\Z)
\quad
\textnormal{ and }
\quad \hat{S}_J^{\pm}(0)=2\del\pm \frac{1}{J+1},$$
and obey the bound
\begin{equation}
\label{defb_j}
|\hat{S}_J^{\pm}(j)|\leq b_j :=  \frac{1}{J+1}+\min\left(2\del, \frac{1}{\pi |j|}\right)
\end{equation}
for all $0\leq |j|\leq J$.

We rewrite our counting function $\mathcal{N}_w(Q, \delta)$ as
$$\mathcal{N}_w(Q, \delta)=
\sum_{\substack{ \a \in \mathbb{Z}^n \\ q \leq Q } } w \left( \frac{\a}{q} \right)
\prod_{r = 1}^R
\chi_{\delta}\left(  q f_r \left( \frac{\a}{q} \right) \right).$$
By using the Selberg magic functions as an upper bound for each of the characteristic functions, we obtain
\begin{eqnarray}
\mathcal{N}_w(Q, \delta)
&=&
\sum_{\substack{ \a \in \mathbb{Z}^n \\ q \leq Q } } w \left( \frac{\a}{q} \right)
\prod_{r = 1}^R
\chi_{\delta}\left(  q f_r \left( \frac{\a}{q} \right) \right)
\\
\notag
&\leq &
\sum_{\substack{ \a \in \mathbb{Z}^n \\ q \leq Q } } w \left( \frac{\a}{q} \right)
\prod_{r = 1}^R
S_J^+\left(  q f_r \left( \frac{\a}{q} \right) \right)
\\
\notag
&=&
\sum_{\substack{ \a \in \mathbb{Z}^n \\ q \leq Q } }w \left( \frac{\a}{q} \right)
\prod_{r=1}^R \left(\sum_{j_r = -J}^{J}
\hat{S}_J^+(j_r)
e \left( j_r q f_r \left( \frac{\a}{q} \right) \right) \right)\\
\notag
&=&
\sum_{\substack{ \a \in \mathbb{Z}^n \\ q \leq Q } }w \left( \frac{\a}{q} \right)
\sum_{ \substack{ 0 \leq |j_1| \leq J  \\ \vdots \\ 0 \leq |j_R| \leq J  }  }
\left(\prod_{r=1}^R\hat{S}_J^+(j_r)\right)
e \left(\sum_{r=1}^R j_r q f_r \left( \frac{\a}{q} \right) \right) \\
&=&
\notag
\sum_{ \substack{ 0 \leq |j_1| \leq J  \\ \vdots \\ 0 \leq |j_R| \leq J  }  }
\left(\prod_{r=1}^R\hat{S}_J^+(j_r)\right)
\sum_{\substack{ \a \in \mathbb{Z}^n \\ q \leq Q } }w \left( \frac{\a}{q} \right)
e \left(\sum_{r=1}^R j_r q f_r \left( \frac{\a}{q} \right) \right).
\end{eqnarray}
The contribution from the terms with $j_1 = \cdots = j_R = 0$  is
\begin{equation*}
\left(2\del+\frac{1}{J+1}\right)^R N_0 = (2\del)^R N_0 + O\left( \del^{R-1}\frac{1}{J} Q^{n+1}+\frac{1}{J^R}Q^{n+1}\right),
\end{equation*}
where the implicit constant may depend on $R$ and an upper bound for the diameter of $\supp w$.
We can obtain a lower bound for $\mathcal{N}_w(Q, \delta)$ in a similar manner.
Therefore, we obtain
\begin{eqnarray}
\label{(3.5)}
\left| \mathcal{N}_w(Q, \delta) - (2\del)^R N_0\right| \ll \del^{R-1}\frac{1}{J} Q^{n+1}+\frac{1}{J^R}Q^{n+1} \\
\notag
+ \sum_{ \substack{ 0 \leq |j_1| \leq J  \\ \vdots \\ 0 \leq |j_R| \leq J\\ \mathbf{j}\neq \mathbf{0}  }  }
\left(\prod_{r=1}^R b_{j_r}\right)\left|
\sum_{\substack{ \a \in \mathbb{Z}^n \\ q \leq Q } } w \left( \frac{\a}{q} \right)
e \left(\sum_{r=1}^R j_r q f_r \left( \frac{\a}{q} \right) \right)\right| .
\end{eqnarray}
By the $n$-dimensional Poisson summation formula we obtain
\begin{eqnarray}
&&\sum_{\substack{ \a \in \mathbb{Z}^n} }
w \left( \frac{\a}{q} \right)
e \left(  \sum_{r=1}^R j_r q f_r \left( \frac{\a}{q} \right)  \right)
\\
\notag
&=&
\sum_{\substack{ \k \in \mathbb{Z}^n } }
\int_{\mathbb{R}^n}
w \left( \frac{\z}{q} \right)
e \left(  \sum_{r=1}^R j_r q f_r \left( \frac{\z}{q} \right)  - \k\cdot \z \right) d\z
\\
\notag
&=&
q^n \sum_{\substack{ \k \in \mathbb{Z}^n } } I(q;  \j; \k ),
\end{eqnarray}
where
$$
I(q; \j; \k ) = \int_{\mathbb{R}^n}
w \left( \x \right)
e \left( \sum_{r=1}^R q  j_r  f_r (\x)  - q \k\cdot \x   \right) d\x.
$$
Therefore, in order to bound the last term in (\ref{(3.5)}) it suffices to obtain an upper bound for
\begin{eqnarray}
\label{N^1}
N^{(r ; \boldsymbol{\epsilon} )}(Q, \delta) = \sum_{ \substack{  1 \leq j_r \leq J  \\ 0 \leq j_s \leq j_r  \\ (s \neq r)     } }
\left(\prod_{r=1}^R b_{j_r}\right)
\left| \sum_{\substack{  q \leq Q } } q^n \sum_{\substack{ \k \in \mathbb{Z}^n  } }  I(q; (\epsilon_1 j_1, \ldots, \epsilon_R j_R) ; \k  ) \right|
\end{eqnarray}
for each $1 \leq r \leq R$ and $\boldsymbol{\epsilon} \in \{-1, 1 \}^{R}$.
Since the arguments are identical after relabeling, we only present the details for bounding $N^{(1; (1, \ldots, 1) )}(Q, \delta)$ in this article;
this is achieved in Section \ref{easycase}. We note that
the same upper bound obtained for $N^{(1; (1, \ldots, 1) )}(Q, \delta)$ holds for $N^{(r ; \boldsymbol{\epsilon} )}(Q, \delta)$ $(1 \leq r \leq R, \boldsymbol{\epsilon} \in \{-1, 1 \}^{R})$ as well.

\section{Proof of Theorem \ref{main upper bound thm}}
\label{easycase}

In this section we obtain an upper bound for $N^{(1; (1, \ldots, 1) )}(Q, \delta)$.
Let us write
$$
I(q; \j; \k ) = \int_{\mathbb{R}^n}
w \left( \x \right)
e \left( q j_1 \left(  F_{\j}( \x)  - \frac{\k\cdot \x}{j_1}  \right) \right) d\x,
$$
where $j_1\neq 0$ and
$$
F_{\j} = f_1 + \frac{j_2}{j_1}f_2 + \cdots + \frac{j_R}{j_1}f_R.
$$ 
Let $\t = (t_2, \ldots, t_R)$ and
$$G(\bfx, \t) = f_1(\x) +\sum_{r=2}^R t_r f_r (\bfx).$$
We apply Lemma \ref{ct1} and \ref{ct2} to the function $G(\bfx, \t)$ and the compact set $[0,1]^{R-1}$. We deduce that there are constants $\tau_{(1; (1, \ldots, 1))} >0$ and $c_1,c_2>0$ such that
\begin{equation}
\label{lowerbdd'}
c_1\leq |\det H_{f_1+\sum_{r=2}^R t_r f_r}(\bfx)|\leq c_2
\end{equation}
for all $\bfx\in B_{2 \tau_{(1; (1, \ldots, 1))}} (\bfx_0)$ and $\bft\in [0,1]^{R-1}$. Moreover, the map
$$\bfx \mapsto \nabla \left(f_1+\sum_{r=2}^R t_r f_r\right) (\bfx)$$
is a $C^{\ell-1}$-diffeomorphism on $\bar{ B_{2 \tau_{(1; (1, \ldots, 1))}} (\bfx_0)}$ for all $\bft\in [0,1]^{R-1}$.
Let us define $\tau_{(r; \boldsymbol{\epsilon})}$ analogously for each case $(r; \boldsymbol{\epsilon})$
$(1 \leq r \leq R, \boldsymbol{\epsilon} \in \{ \pm 1\}^R)$ and let
$$
0 < \tau \leq  \min_{\substack{ 1 \leq r \leq R \\ \boldsymbol{\epsilon} \in \{ \pm 1\}^R }}  \tau_{(r; \boldsymbol{\epsilon})}.
$$
We let $\tau$ be sufficiently small such that Lemma \ref{lem more lemma} is applicable, and 
choose $\varepsilon_0 > 0$ in the statement of Theorem \ref{main upper bound thm}
to be smaller than $2 \tau$.
By Lemma \ref{ct2} it follows
that there exist constants $0 < \kap < \tau$ and $\rho >0$ such that
\begin{eqnarray}
\notag
\textnormal{dist} \left( \partial  \left( \nabla ( f_1  + \sum_{r=2}^Rt_r f_r ) (B_{\tau}(\bfx_0)) \right), \partial  \left( \nabla ( f_1  + \sum_{r=2}^R t_r f_r ) (B_\kap (\bfx_0)) \right)  \right)
\geq 2 \rho
\\
\label{defnrho}
\end{eqnarray}
for all $\bft\in [0,1]^{R-1}$.

Let $\calD = B_\tau (\bfx_0)$ and let $w \in C_0^{\infty}(\mathbb{R}^n)$
be a non-negative weight function such that 
$$
U := \supp w = \{ \x \in \mathbb{R}^n :   w(\x) \not = 0  \} \subseteq B_\kap(\bfx_0).
$$
We define
$$
V_{\j} = \nabla F_{\j} (U)
\quad
\textnormal{ and }
\quad
\mathcal{R}_{\j} = \nabla F_{\j} (\mathcal{D}).
$$
Since $0 \leq j_r / j_1 \leq 1$ $(2 \leq r \leq R)$, we know that $\nabla F_{\j}$ is a diffeomorphism on $U$ and $\mathcal{D}$. Let $L \in \mathbb{N}$ be such that
\begin{eqnarray}
\label{defL}
V_{\j} \subseteq [- L, L]^{n}
\end{eqnarray}
for all $1 \leq j_1 \leq J$, $0 \leq  j_2, \ldots, j_R \leq  j_1$. Note that we can choose $L$ independently of $J$.\par
We split the set of $\k \in \ZZ^n$ into three disjoint subsets.
Let
$$
\mathcal{K}_{\j; 1} = \{ \k \in \mathbb{Z}^n :  \frac{\k}{j_1} \in V_{\j} \},
$$
$$
\mathcal{K}_{\j; 2} = \{ \k \in \mathbb{Z}^n : \textnormal{dist} \left( \frac{\k}{j_1}, V_{\j} \right) \geq \rho \}
$$
and
$$
\mathcal{K}_{\j; 3} = \mathbb{Z}^n \setminus (\mathcal{K}_{\j; 1} \cup \mathcal{K}_{\j; 2}).
$$
For each $1 \leq i \leq 3$, we let
\begin{eqnarray}
N_i =
\sum_{\substack{  \\ 1 \leq j_1 \leq J \\  0 \leq  j_2, \ldots, j_R \leq  j_1  }}
\left(\prod_{r=1}^R b_{j_r}\right)
\left|  \sum_{\substack{  q \leq Q } } q^n \sum_{\substack{ \k \in \mathcal{K}_{\j; i}  } } I(q; \j; \k) \right|
\end{eqnarray}
so that
\begin{eqnarray}
N^{(1; (1, \ldots, 1) )}(Q, \delta) \ll  N_1 + N_2 + N_3.
\end{eqnarray}

We now bound each $N_i$ separately.
\subsection{Case $\k \in \mathcal{K}_{\j; 2}$}
\label{sec5.1}
In this case we let
$$
\phi_1(\x) =  \frac{ j_1  f_1 (\x) +  \cdots  +  j_R f_R ( \x)  - \k\cdot \x}{\textnormal{dist} (\k, j_1 V_{\j} )},
$$
where $\k\cdot \x = k_1 x_1 + \cdots + k_n x_n$,
and
$$
\lambda_1 = q \cdot \textnormal{dist} (\k, j_1 V_{\j} ).
$$
Then it follows from the definition of $V_{\j}$ that
$$
|\nabla \phi_1 (\x)|  = \frac{ |  j_1  \nabla f_1 (\x) +  \cdots  +   j_R \nabla f_R ( \x)  - \k  | }{\textnormal{dist} (\k, j_1 V_{\j} )} \geq 1
\quad
(\x \in U).
$$
Let $U_+ \subseteq \RR^n$ be an open set such that $\overline{U} \subseteq U_+ \subseteq B_{\tau} (\x_0)$,
$$
V_{\j +} = \nabla F_{\j} (U_+) \subseteq [-2L, 2L]^n \quad (1 \leq j_1 \leq J, 0 \leq j_2, \ldots, j_R \leq j_1),
$$
\begin{eqnarray}
\min_{\y \in \overline{U}}  |\x - \y|
\label{Vj+}
< \frac{ \dist (\partial \mathcal{D}, \partial U) }{4} \quad (\x \in U_+),
\end{eqnarray}
and
$$
\min_{\y \in \overline{U}} \max_{ \t \in [0,1]^{R-1} } \left| \nabla (f_1 + \sum_{r=2}^R t_r f_r) (\x) - \nabla (f_1 + \sum_{r=2}^R t_r f_r) (\y) \right|
< \frac{\rho}{2} \quad (\x \in U_+).
$$
Thus we obtain
$$
|\nabla \phi_1 (\x)|  \geq \frac12
\quad
(\x \in U_+).
$$

Next we need upper bounds for the derivatives of $\phi_1$.
\begin{lem}\label{Claim1}
Given $i_1, \ldots, i_n \in \ZZ_{\geq 0}$ with $\sum_{m=1}^n i_m\leq \ell$ we have
$$
\left| \frac{\partial^{i_1 + \cdots + i_n}  \phi_1}{ \partial x_1^{i_1} \cdots \partial x_n^{i_n} }  (\x)  \right| \ll 1 \quad  (\x \in U_+),
$$
where the implicit constant depends only on $(i_1,\ldots, i_n)$, $\rho$ and
upper bounds for (the absolute values of) finitely many derivatives of $f_r$ $(1\leq r\leq R)$ on $U_+$. In particular, it is independent of $\j$ and $\k.$ 
\end{lem}

\begin{proof}
First we show that $|\phi_1 (\x)|$ is bounded independently of the choice of $\j$ and $\k$.
To see this we let $C > 0$ be such that
$$
\frac{1}{C} \max_{ \substack{ \t \in [0,1]^{R-1} \\ \mathbf{x} \in \overline{U}  }  } \   | \nabla (f_1 + t_2 f_2 + \cdots + t_R f_R ) (\x) | < \frac12.
$$
Suppose $|\k| \geq C j_1$. Then we have
$$
\left| \frac{\k}{|\k|}- \frac{j_1 \y}{|\k|} \right| \geq
1 - \frac{|j_1 \y|}{|\k|} > \frac12 \quad (\y \in V_{\j}),
$$
and hence,
$$
\textnormal{dist} \left( \frac{\k}{|\k|}, \frac{j_1 V_{\j}}{|\k|} \right) \geq
\frac12.
$$
Therefore, it follows that
$$
|\phi_1(\x)| =\left| \frac{ \frac{j_1}{|\k|}  f_1 (\x) + \cdots  + \frac{j_R}{|\k|} f_R ( \x)  - \frac{\k\cdot \x}{|\k|}}{\textnormal{dist} (\frac{\k}{|\k|}, \frac{j_1 V_{\j}}{|\k|} )} \right| \ll 1 \quad (\x \in U_+),
$$
where the implicit constant is independent of $\j$  (since $j_1 \geq 1$  and $0 \leq  j_2, \ldots, j_R \leq j_1$) and $\k$. On the other hand, suppose $|\k| < C j_1$.
Then by the definition of $\mathcal{K}_{\j; 2}$ we have
$$
|\phi_1(\x)| \leq \left| \frac{  f_1 (\x) + \frac{j_2}{j_1} f_2 ( \x) + \cdots + \frac{j_R}{j_1} f_R ( \x) - \frac{\k \cdot \x}{j_1} }{\rho} \right| \ll 1
\quad (\x \in U_+),
$$
where the implicit constant is independent of $\j$ and $\k$.
For the first partial derivatives, we can obtain the same conclusion by a similar argument. For higher partial derivatives,
the term $\k \cdot \x$ of $\phi_1(\x)$ vanishes and the desired conclusion can be deduced easily.
\end{proof}
Therefore, it follows from Lemma \ref{Lem1} (with $\phi_1$ and $\lambda_1$) that
$$
I(q; \j; \k) \ll \lambda_1^{- \ell + 1} = (q \cdot \textnormal{dist} (\k, j_1 V_{\j} ))^{- \ell + 1},
$$
where the implicit constant is independent of $\j$ and $\k$.
As a result,  since $\ell - 1 - n \geq  1$ we obtain
\begin{eqnarray}
\label{(4.4)}
\sum_{\k \in \mathcal{K}_{\j; 2}} I(q; \j; \k) &\ll&   q^{-\ell + 1} \sum_{\k \in \mathcal{K}_{\j; 2}} \textnormal{dist} (\k, j_1 V_{\j} )^{- \ell + 1}
\\
\notag
&\leq&  q^{-\ell + 1} \sum_{d = 0}^{\infty} \sum_{
\substack{
\k \in \mathbb{Z}^n
\\
 2^d j_1 \rho  \leq \textnormal{dist} (\k, j_1 V_{\j} )
< 2^{d+1} j_1 \rho
}
  } \frac{ 1 }{(2^d j_1 \rho)^{\ell-1}}
\\
&\ll&  q^{-\ell + 1} \sum_{d = 0}^{\infty}  \frac{(L  j_1 + 2^{d+1} j_1 \rho)^{n}  }{(2^{d} j_1 \rho)^{\ell-1}}
\notag
\\
\notag
&\ll &   q^{-\ell + 1}.
\end{eqnarray}
Here the implicit constants depend only on $L$, $\rho$, $n$, $\ell$ and upper bounds for (the absolute values of) finitely many derivatives of $w$ and $f_r$ $(1\leq r\leq R)$ on $U_+$, but it does not depend on $j_1$.
Therefore, it follows that
\begin{eqnarray}
\label{finalbddN2}
N_2 &\ll& \sum_{\substack{1\leq j_1\leq J\\ 0\leq j_2,\ldots, j_R\leq j_1}}  \left(\prod_{r=1}^R \left(\frac{1}{J} +\min \left(\del, \frac{1}{j_r}\right)\right) \right)  \sum_{q\leq Q} q^{n- \ell +1} \\
\notag
& \ll & \log Q \left(\sum_{0\leq j\leq J}  \left(\frac{1}{J} +\min \left(\del, \frac{1}{j}\right) \right)\right)^R \\
\notag
& \ll &  \log Q \left( 1 + \log J\right)^R.
\notag
\end{eqnarray}

\subsection{Case $\k \in \mathcal{K}_{\j; 3}$}
\label{seck3}
Let $\lambda = q j_1$ and
$$
\phi(\x) = F_{\j} (\x)  - \frac{\k}{j_1} \cdot \x.
$$
Since $\nabla F_{\j}$ is a diffeomorphism on $\mathcal{D}$, it follows that for each fixed $\j$ we have that each $\k \in j_1 \mathcal{R}_{\j}$
determines a unique critical point $\x_{\j; \k}$ of $\phi$, i.e.
$$
\mathbf{0} =  \nabla \phi (\x_{\j; \k}) =
\nabla F_{\j} (\x_{\j;\k})  - \frac{\k}{j_1}.
$$
Then by (\ref{dual}) we have
\begin{eqnarray}
\label{(4.5)-1}
\phi(\x_{\j; \k}) = F_{\j}(\x_{\j;\k}) - \frac{\k}{j_1} \cdot \x_{\j;\k} = - F^*_{\j}\left( \frac{\k}{j_1} \right),
\end{eqnarray}
which we will make use of in the next section.
It also follows from the definition of $\mathcal{R}_{\j}$ that
$$
\x_{\j; \k} = ( \nabla F_{\j} )^{-1} (\k / j_1) \in \mathcal{D}.
$$
Let $\mathcal{D}_+$ be an open set such that $\overline{\mathcal{D}} \subseteq \mathcal{D}_+ \subseteq B_{3 \tau /2} (\x_0)$.
Recall we have set $\tau > 0$ to be sufficiently small. 
\begin{lem}
\label{lem more lemma}
We have
$$
\frac{| \x - \x_{\j;\k} |}{| \nabla \phi (\x) |} \ll 1 \quad (\x \in \mathcal{D}_+, \x\neq \x_{\j;\k} ),
$$
where the implicit constant is independent of $\j$ and $\k$.
\end{lem}
\begin{proof}
Since
$$
\frac{| \x - \x_{\j;\k} |}{| \nabla \phi (\x) |}
=
\frac{| \x - ( \nabla F_{\j} )^{-1} (\k / j_1) |}{| \nabla F_{\j} (\x) - \frac{\k}{j_1} |}
$$
and $\k/j_1 \in \mathcal{R}_{\j}$,
it suffices to prove
\begin{equation}
\label{eqnlemmamore}
1 \ll \frac{| \x - \y |}{| \nabla F_{\j} (\x) - \nabla F_{\j} (\y) |} \ll 1 \quad (\x, \y \in \overline{\mathcal{D}_+}, \x\neq \y),
\end{equation}
where the implicit constant is independent of $\j$; the lower bound is not necessary here, but it will be needed in Section \ref{auxinequ}.
By Taylor's theorem we have
\begin{eqnarray}
\nabla F_{\j} (\x) - \nabla F_{\j} (\y) = H_{F_{\j}}(\y) \cdot (\x - \y) + O( | \x - \y|^2 ),
\label{ineq11}
\end{eqnarray}
where the implicit constant is independent of $\j$. Recall that given an invertible real symmetric $n \times n$ matrix $A$ we have
$$
\lambda_{\min}|\z| \ll_n |A \cdot \z| \quad (\z \in \RR^n),
$$
where $\lambda_{\min}$ is the minimum of the absolute values of the eigenvalues of $A$.
Thus from (\ref{ineq11}), (\ref{lowerbdd'}) and the fact that the eigenvalues of a matrix are continuous in the coefficients of the matrix,
it follows that there exist constants $C_1, \lambda_1 > 0$,
which are independent of $\j$, such that
\begin{eqnarray}
\label{eqn1lemma}
|\x - \y| \leq C_1 |\nabla F_{\j} (\x) - \nabla F_{\j} (\y)|
\end{eqnarray}
for all $| \x - \y | \leq \lambda_1$. The lower bound in (\ref{eqnlemmamore}) follows directly from (\ref{ineq11}).
\end{proof}
Since $\nabla F_{\j} (\x_{\j;\k}) = \k/j_1 \not \in V_{\j}$ we have $\x_{\j;\k} \not \in U = \supp w,$ i.e. $w(\x_{\j;\k}) = 0$.
The cardinality of $\mathcal{K}_{\j; 3}$ can be bounded as follows
\begin{eqnarray}
\label{cardK3-1}
\# \mathcal{K}_{\j; 3}  \leq
\sum_{  \substack{  \k \in \mathbb{Z}^n \\  \textnormal{dist} (\k/j_1, V_{\j} ) < \rho }} 1
\leq \sum_{  \substack{  \k \in \mathbb{Z}^n \\  \k/j_1 \in [-L  -\rho, L   + \rho]^n }} 1 \ll_{L, \rho} j_1^n.
\end{eqnarray}
Given $i_1, \ldots, i_n \in \ZZ_{\geq 0}$ with $\sum_{m=1}^n i_m \leq \ell$ we have
\begin{eqnarray}
\label{partials1}
\left| \frac{\partial^{i_1 + \cdots + i_n}  \phi}{ \partial x_1^{i_1} \cdots \partial x_n^{i_n} }  (\x)  \right| \ll 1 \quad (\x \in \mathcal{D}_+),
\end{eqnarray}
where the implicit constant depends only on $(i_1,\ldots, i_n)$, $\rho$ and
upper bounds for (the absolute values of) finitely many derivatives of $f_r$ $(1\leq r\leq R)$ on $\mathcal{D}_+$.
In particular, it is independent of $\j$ and $\k$.
We also have  $H_{\phi} = H_{F_{\j}}$.
Therefore, it follows from Lemma \ref{Lem2}, (\ref{lowerbdd'}) and (\ref{cardK3-1}) that
\begin{eqnarray}
\sum_{\k \in \mathcal{K}_{\j; 3} }
I(q; \j ; \k) \ll  j_1^n \lambda^{-\frac{n}{2} - 1} =  q^{-\frac{n}{2} - 1} j_1^{ \frac{n}{2} - 1 }.
\end{eqnarray}
Note that here we used $\ell > \frac{n}{2}+4$.
Consequently, by a similar calculation as in (\ref{finalbddN2}) we obtain
\begin{equation}
\label{finalbddN3}
N_3 \ll   J^{ \frac{ n }{ 2 } - 1} Q^{ \frac{ n }{ 2 } } (1+\log J)^{R}.
\end{equation}

\subsection{Case $\k \in \mathcal{K}_{\j; 1}$}
Let $\phi$ and $\lambda$ be as in Section \ref{seck3}. In particular, we have $(\ref{(4.5)-1})$ and Lemma \ref{lem more lemma}.
It can be verified that $\phi$ and its partial derivatives satisfy (\ref{partials1}).
The signature of $H_{\phi}(\x_{\j;\k}) = H_{F_{\j}}(\x_{\j;\k})$ is constant for all choices of $\j$ and $\k$ in consideration;
this follows from (\ref{lowerbdd'}) and the fact that the eigenvalues of a matrix are continuous in the coefficients of the matrix.
Let us denote by $\sigma$ the signature of $H_{\phi}(\x_{\j;\k})$. Then it follows from Lemma \ref{Lem2} and (\ref{lowerbdd'}) that
\begin{eqnarray}
I(q; \j; \k) &=& \frac{w(\x_{\j;\k})}{  \sqrt{ |  \det H_{ F_{\j} } (\x_{\j;\k})  | } } (q j_1)^{-\frac{n}{2}}
e \left(  - q j_1 F_{\j}^* \left(  \frac{\k}{j_1}  \right)   +   \frac{\sigma}{8} \right)
\\
\notag
&+&   O\left(  (q j_1)^{-\frac{n}{2} -1 }) \right),
\end{eqnarray}
where the implicit constant is independent of $\j$ and $\k$.
Let $w_{\j}^* = w \circ (\nabla F_{\j})^{-1}$.
Hence, for $\k \in \mathcal{K}_{\j; 1}$ we have by partial summation
\begin{eqnarray}
&&\sum_{q \leq Q} q^{n} I(q; \j; \k)
\\
\notag
&\ll&
\frac{w_{\j}^* \left( \frac{\k}{j_1} \right)}{  \sqrt{ |  \det H_{ F_{\j} } (\x_{\j;\k})  | }   }   j_1^{ -\frac{ n}{2}  } Q^{\frac{n}{2} } \min ( \|  j_1 F_{\j}^*(\k/j_1)  \|^{-1}   , Q )
+  j_1^{-\frac{n}{2} - 1} Q^{\frac{n}{2}}.
\end{eqnarray}
Consequently, we obtain
\begin{eqnarray}
&&N_1
\label{N1 firstbdd}
\\
\notag
&\ll& Q^{ \frac{n}{2}  }
\sum_{\substack{   1 \leq  j_1  \leq J  \\  0 \leq  j_2, \ldots, j_R \leq  j_1    }} \left( \prod_{r=1}^R b_{j_r}\right)
\sum_{\substack{ \k \in \mathcal{K}_{\j; 1}  } } \frac{w_{\j}^* \left( \frac{\k}{j_1} \right)}{  \sqrt{ |  \det H_{ F_{\j} } (\x_{\j;\k})  | }   } j_1^{ -\frac{n}{2} }  \min ( \|  j_1 F_{\j}^*(\k/j_1)  \|^{-1}   , Q )
\\
\notag
&+&
 Q^{\frac{n}{2}}
 \sum_{\substack{  1 \leq  j_1  \leq J  \\  0 \leq  j_2, \ldots, j_R \leq  j_1   }}
\prod_{r=1}^R \left(\frac{1}{J} +\min \left(\del, \frac{1}{j_r}\right)\right)
\sum_{\substack{ \k \in \mathcal{K}_{\j; 1}  } } j_1^{-\frac{n}{2} - 1}.
\end{eqnarray}
By a similar argument as in (\ref{cardK3-1}) we have
$$
\# \mathcal{K}_{\j; 1} \ll_L  j_1^n.
$$
Thus the second term in (\ref{N1 firstbdd}) can be bounded by
$$
\ll
Q^{\frac{n}{2}}  J^{\frac{n}{2}  - 1} (1+\log J)^{R}.
$$
Recall that by definition $\k \in \mathcal{K}_{\j;1}$ means $\k \in j_1 V_{\j}$.
Therefore, by simplifying the first term in (\ref{N1 firstbdd}) we obtain
\begin{eqnarray}
\label{N1lastbdd}
N_1& \ll& Q^{  \frac{n}{2} + 1}
\sum_{\substack{  1 \leq  j_1  \leq J  \\  0 \leq  j_2, \ldots, j_R \leq  j_1   }} \left( \prod_{r=1}^R b_{j_r}\right)
\sum_{\substack{  \substack{  \k \in j_1 V_{\j}  \\  \|  j_1 F_{\j}^*(\k/j_1)  \| < Q^{-1}   }   } } w_{\j}^* \left( \frac{\k}{j_1} \right) j_1^{- \frac{ n}{2} }
\\
\notag
&+&
Q^{  \frac{n}{2} }
\sum_{\substack{  1 \leq  j_1  \leq J  \\  0 \leq  j_2, \ldots, j_R \leq  j_1    }} \left( \prod_{r=1}^R b_{j_r}\right)
\sum_{\substack{  \k \in j_1 V_{\j}  \\   \|  j_1 F_{\j}^*(\k/j_1)  \| \geq Q^{-1}    } } w_{\j}^* \left( \frac{\k}{j_1} \right) j_1^{  -\frac{n}{2}  }  \|  j_1 F_{\j}^*(\k/j_1)  \|^{-1}
\\
\notag
&+& Q^{\frac{n}{2}}  J^{\frac{n}{2}  - 1} (1+\log J)^{R}.
\end{eqnarray}

We present the proof of the following result in Section \ref{auxinequ}.
\begin{prop}
\label{prop main}
Let $T >0$ and $J_2,\ldots, J_R \in [1, J]$. Then with the notations from this section, we have
$$
\sum_{\substack{   1 \leq  j_1  \leq J  \\  0 \leq  j_r \leq \min\{J_r, j_1\} \\ (2\leq r\leq R)    }}
\sum_{\substack{  \substack{  \k \in j_1 V_{\j}  \\  \|  j_1 F_{\j}^*(\k/j_1)  \| < T^{-1}   }   } } w_{\j}^* \left( \frac{\k}{j_1} \right)
\ll T^{-1} J^{n+1}\left( \prod_{r=2}^R J_r \right)+ \left( \prod_{r=2}^R J_r\right) J^n \mathcal{E}_n(J),
$$
where
$$
\mathcal{E}_n(J)
=
\mathcal{E}^{(\mathfrak{c}'_1; \mathfrak{c}'_2)}_n(J)
=
\begin{cases}
    \exp( \mathfrak{c}'_1 \sqrt{\log J}) & \mbox{if } n = 2 \\
    (\log J)^{\mathfrak{c}'_2} & \mbox{if } n \geq 3 \\
\end{cases}
$$
for some positive constants $\mathfrak{c}'_1$ and $\mathfrak{c}'_2.$ Here the constants $\mathfrak{c}'_1$ and $\mathfrak{c}'_2$ and the implicit constants depend only on $n$, $R$, $c_1$ and $c_2$ in (\ref{lowerbdd'}), $\rho$ in (\ref{defnrho}), $\rho'$ in (\ref{defrho'})
and upper bounds for (the absolute values of) finitely many derivatives of $w$ and $f_r$ $(2\leq r\leq R)$ on $\mathcal{D}_+$.
In particular, the implicit constant is independent of $T$ and $J_2, \ldots, J_R$.
\end{prop}

We recall that $b_j \ll \frac{1}{j}$ $(1 \leq j \leq J)$.
Let $\mathfrak{I}_{0} = \{ 0 \}$ and $\mathfrak{I}_s = [2^{s-1}, 2^{s}]$ $(s \in \NN)$.
Then it follows from Proposition \ref{prop main} and partial summation that
\begin{eqnarray}
\label{partial111}
&&\sum_{\substack{   1 \leq  j_1  \leq J  \\  0 \leq  j_2, \ldots, j_R \leq  j_1    }}  \left( \prod_{r=1}^R b_{j_r}\right) j_1^{-\frac{n}{2}}
\sum_{\substack{  \substack{  \k \in j_1 V_{\j}  \\  \|  j_1 F_{\j}^*(\k/j_1)  \| < T^{-1}   }   } } w_{\j}^* \left( \frac{\k}{j_1} \right)
\\
&\ll&
\notag
\sum_{0 \leq s_2, \ldots, s_R \leq \frac{\log J}{\log 2} + 1 } \sum_{\substack{   1 \leq  j_1  \leq J  \\  j_r \in \mathfrak{I}_{s_r} \cap [0, j_1] \\ (2 \leq r \leq R)     }}  \left( \prod_{r=2}^R 2^{- s_r}\right) j_1^{-\frac{n}{2}-1}
\sum_{\substack{  \substack{  \k \in j_1 V_{\j}  \\  \|  j_1 F_{\j}^*(\k/j_1)  \| < T^{-1}   }   } } w_{\j}^* \left( \frac{\k}{j_1} \right)
\\
&\ll&
\notag
J^{-\frac{n}{2}-1} (1+\log J)^{R} (T^{-1} J^{n+1} +  J^n \mathcal{E}_n(J)).
\end{eqnarray}
For the first sum in (\ref{N1lastbdd}), we set $T = Q$ in (\ref{partial111}).
For the second sum in (\ref{N1lastbdd}), we split the interval $[Q^{-1}, 1/2]$ into dyadic intervals.
Since the sum is a trivial sum if $Q^{-1}> 1/2$, we assume $Q^{-1} \leq 1/2$.
Then by (\ref{partial111}) we have
\begin{eqnarray}
\label{partialll!!}
&& \sum_{\substack{  1 \leq  j_1  \leq J  \\  0 \leq  j_2, \ldots, j_R \leq  j_1     }} \left( \prod_{r=1}^R b_{j_r}\right)   \sum_{ \substack{ \k \in  j_1 V_{\j}  \\  \|  j_1 F^*_{\j}(\k/j_1)  \| \geq Q^{-1} } }
w^* \left(  \frac{\k}{j_1}  \right) j_1^{-\frac{n}{2}}
\|  j_1 F^*_{\j} ( {\k}/{j_1}  )  \|^{-1}
\\
\notag
&\leq& \sum_{1 \leq i \leq \frac{\log Q}{\log 2} + 1 } Q 2^{1 - i}
\sum_{\substack{ 1 \leq  j_1  \leq J  \\  0 \leq  j_2, \ldots, j_R \leq  j_1    }} \left( \prod_{r=1}^R b_{j_r}\right)
\sum_{ \substack{   \k \in j_1 V  \\
\frac{2^{i-1}}{Q}  \leq \|  j_1 F^*_{\j}(\k/j_1)  \| \leq   \frac{2^{i}}{Q} } }
w^* \left(  \frac{\k}{j_1}  \right) j_1^{-\frac{n}{2}}
\\
\notag
&\ll& \sum_{1 \leq i \leq \frac{\log Q}{\log 2} + 1 } Q 2^{1 - i}
 J^{-\frac{n}{2} -1} (1 + \log J)^{R}   (2^i Q^{-1} J^{n+1} + J^n \mathcal{E}_n(J))
\\
&\ll&
\notag
\left( (\log Q) J^{\frac{n}{2} } +  Q   J^{ \frac{n}{2} -1}  \mathcal{E}_n(J) \right) (1 + \log J)^{R}.
\end{eqnarray}
Therefore, we obtain from (\ref{N1lastbdd}), (\ref{partial111}) and (\ref{partialll!!}) that  
\begin{eqnarray}
\label{N1lastbdd+}
N_1& \ll& Q^{  \frac{n}{2} + 1}   J^{-\frac{n}{2}-1}  (Q^{-1} J^{n+1} +  J^n \mathcal{E}_n(J)) (1 + \log J)^{R}
\\
\notag
&+&
Q^{\frac{n}{2} }   \left( (\log Q)    J^{\frac{n}{2} } + Q     J^{\frac{n}{2} -1}  \mathcal{E}_n(J) \right) (1 + \log J)^{R}
\\
\notag
&+&  Q^{\frac{n}{2}}  J^{\frac{n}{2}  - 1} (1 + \log J)^{R}
\\
&\ll&
\notag
\left( (\log Q)  Q^{  \frac{n}{2} }  J^{\frac{n}{2} }  +  Q^{  \frac{n}{2} + 1}   J^{\frac{n}{2} - 1 }  \mathcal{E}_n(J)
\right) (1 + \log J)^{R}.
\end{eqnarray}
Recall $n\geq 2$. Combining  (\ref{finalbddN2}), (\ref{finalbddN3}) and  (\ref{N1lastbdd+}) yields
\begin{eqnarray*}
&&N^{(1; (1, \ldots, 1) )}(Q, \delta)
\notag
\\
& \ll& (\log Q) (1 + \log J)^R +   J^{ \frac{ n }{ 2 } - 1} Q^{ \frac{ n }{ 2 } } (1+\log J)^R
\\
&+& \left( (\log Q)  Q^{  \frac{n}{2} }  J^{\frac{n}{2} }  +  Q^{  \frac{n}{2} + 1}   J^{\frac{n}{2} - 1 }  \mathcal{E}_n(J) \right) (1+\log J)^R
\\
&\ll&
\left( (\log Q) Q^{  \frac{n}{2} }  J^{\frac{n}{2} }  +   Q^{  \frac{n}{2} + 1}   J^{\frac{n}{2} - 1 }  \mathcal{E}_n(J) \right) (1+\log J)^R.
\end{eqnarray*}
With this bound in hand, it follows from (\ref{(3.5)}) and (\ref{N^1}) (recall the remark made after (\ref{N^1})) that
\begin{eqnarray}\label{eqn527}
&&\left| \mathcal{N}_w(Q, \delta) - (2\del)^R N_0\right|
\\
&\ll& \del^{R-1}\frac{1}{J} Q^{n+1}+\frac{1}{J^R}Q^{n+1}
\notag
+ (\log Q) (1 + \log J)^{R} Q^{  \frac{n}{2} }  J^{\frac{n}{2} }
\\
\notag
&+& (1 + \log J)^R  Q^{  \frac{n}{2} + 1}   J^{\frac{n}{2} - 1 }  \mathcal{E}_n(J).
\end{eqnarray}

We are still free to choose the parameter $J\geq 1$. For $R\geq 1$ we have
\begin{eqnarray}\label{four relations}
Q^{\frac{n}{2}}J^{\frac{n}{2}} < Q^{\frac{n}{2}+1}J^{\frac{n}{2}-1} &\Leftrightarrow& J< Q
\\ 
\frac{1}{J^R} Q^{n+1} \leq Q^{\frac{n}{2}+1}J^{\frac{n}{2}-1} &\Leftrightarrow & Q^{\frac{n}{n+2(R-1)}} \leq  J
\notag 
\\
\notag
\del^{R-1} J^{-1} Q^{n+1} < \frac{1}{J^R} Q^{n+1} &\Leftrightarrow& J < \del^{-1}\\
\notag
 \del^{R-1}J^{-1}Q^{n+1} \leq Q^{\frac{n}{2}+1}J^{\frac{n}{2}-1} &\Leftrightarrow & Q\del^{\frac{2(R-1)}{n}} \leq  J.
\end{eqnarray}
We now make a case distinction.
We assume $Q \geq 2$ as Theorem \ref{main upper bound thm} when $Q = 1$ follows immediately from (\ref{(3.5)}).
If $\del^{-1} > Q^{\frac{n}{n+2(R-1)}}$ then we set $J=Q^{\frac{n}{n+2(R-1)}}$, and  we obtain from
the term $(1 + \log J)^R Q^{  \frac{n}{2} + 1}   J^{\frac{n}{2} - 1 }  \mathcal{E}_n(J)$
(see the first three equivalences in (\ref{four relations}))
that
\begin{eqnarray}
\left| \mathcal{N}_w(Q, \delta) - (2\del)^R N_0\right| &\ll& Q^{n + 1-\frac{Rn}{n+2(R-1)}} (\log Q)^{R} \mathcal{E}_n(Q)
\notag
\\
\notag
&=& Q^{n -\frac{(n-2) (R-1) }{n + 2(R-1)}} (\log Q)^{R} \mathcal{E}_n(Q).
\end{eqnarray}
If $\del^{-1} \leq Q^{\frac{n}{n+2(R-1)}}$ then we set $J=Q\del^{\frac{2(R-1)}{n}}$, and  we obtain from
the term $(1 + \log J)^R Q^{  \frac{n}{2} + 1}   J^{\frac{n}{2} - 1 }  \mathcal{E}_n(J)$
(see the first and the last two equivalences in (\ref{four relations}))
that
\begin{eqnarray}
\left| \mathcal{N}_w(Q, \delta) - (2\del)^R N_0\right| &\ll& \del^{\frac{(R-1)(n-2)}{n}} (\log Q)^{R} Q^n  \mathcal{E}_n(Q).
\notag
\end{eqnarray}
Finally, we note that
$$
\mathcal{E}_n(Q) = \mathcal{E}^{(\mathfrak{c}'_1; \mathfrak{c}'_2)}_n(Q) \ll \mathcal{E}^{(\mathfrak{c}'_1 + c_0 R; \mathfrak{c}'_2 + R)}_n(Q)
$$
for some absolute constant $c_0 > 0$.
Therefore, we let $\mathfrak{c}_1 = \mathfrak{c}'_1 + c_0 R$ and $\mathfrak{c}_2 = \mathfrak{c}'_2 + R$, and
this completes the proof of Theorem \ref{main upper bound thm}.

\section{Proof of Proposition \ref{prop main}}\label{auxinequ}
\label{input1}

Recall the definitions $F_{\j} = f_1 + \frac{j_2}{j_1} f_2 + \cdots + \frac{j_R}{j_1} f_R$,
$w_{\j}^* = w \circ (\nabla F_{\j})^{-1}$, $U = \supp w$ and $V_{\j} = \nabla F_{\j}(U)$.
It is clear that the result when $0 < T < 2$ follows from the case $T = 2$.
Let $T \geq 2$ and $D = \lfloor T / 2  \rfloor$. We define
\begin{eqnarray}
\label{eqn 6.1}
\mathcal{M} ( J, T^{-1}) &=& \sum_{ \substack{ 1 \leq  j_1  \leq J  \\  0 \leq  j_r \leq \min\{J_r, j_1\} \\ (2\leq r\leq R) }}
\sum_{\substack { \a \in  j_1 V_{\j}  \\  \|  j_1 F_{\j}^*(\a/j_1)  \| \leq T^{-1} } }
w_{\j}^* \left(  \frac{\a}{j_1}  \right)
\\
\notag
&=& \sum_{ \substack{ 1 \leq  j_1  \leq J  \\ 0 \leq  j_r \leq \min\{J_r, j_1\} \\ (2\leq r\leq R) } }
\sum_{\substack { \a \in  \ZZ^n  \\  \|  j_1 F_{\j}^*(\a/j_1)  \| \leq T^{-1} } }
w_{\j}^* \left(  \frac{\a}{j_1}  \right).
\end{eqnarray}
Next we consider the Fej\'{e}r kernel $\mathcal{F}_{D}$. By the trigonometric identity
$1 - 2 \sin^2 (x) = \cos(2 x)$ $(x \in \mathbb{R})$ it follows that
\begin{eqnarray}
\label{FeKe1}
\mathcal{F}_{D}(\theta) =
D^{-2} \left|  \sum_{d=1}^{D} e(d \theta) \right|^2
= \left(  \frac{\sin (\pi D \theta) }{ D \sin (\pi \theta) }  \right)^2
= \sum_{d = -D}^{D} \frac{D - |d|}{D^2} e(d \theta).
\end{eqnarray}
Since the sine function is concave on $[0, \pi/2]$, we have
$$\sin (x) \geq 2\pi^{-1}x$$
for $x \in [0, \pi/2]$.
Thus it can be verified easily that
$$
\left(  \frac{\sin (\pi D \theta) }{ D \sin (\pi \theta) }  \right)^2
\geq \left(  \frac{2 \pi^{-1} \pi D \| \theta \|}{ D \pi \|  \theta \| }  \right)^2 \geq \frac{4}{\pi^2}
$$
when $0 < \| \theta \| \leq  T^{-1}$; therefore, recalling the definition of $\chi_{T^{-1}}$ in (\ref{defchi}) it follows that
\begin{eqnarray}
\label{C-Fbdd}
\chi_{T^{-1}}(\theta) \leq \frac{\pi^2}{4} \cF_{D}(\theta) =  \frac{\pi^2}{4}\sum_{d = -D}^{D} \frac{D - |d|}{D^2} e(d \theta) \quad (\theta \in \mathbb{R}).
\end{eqnarray}
By inserting (\ref{C-Fbdd}) into (\ref{eqn 6.1}), we obtain
\begin{eqnarray}
\label{sum 6.2}
\mathcal{M} ( J, T^{-1}) \leq \frac{\pi^2}{4}
\sum_{ \substack{ \a \in \mathbb{Z}^n  \\  1 \leq  j_1  \leq J  \\ 0 \leq  j_r \leq \min\{J_r, j_1\} \\ (2\leq r\leq R)   } }
\sum_{ d = - D}^D \frac{D - |d|}{D^2}
w_{\j}^* \left(  \frac{\a}{j_1}  \right) e\left( d j_1 F^*_{\j} \left(  \frac{\a}{j_1}  \right)  \right).
\end{eqnarray}
Since $\supp w = U$ it follows that
$$
\supp w^*_{\j} = V_{\j} \subseteq [-L, L]^n,
$$
where $L$ is defined in (\ref{defL}). Recall $J_r\geq 1$ $(2\leq r\leq R)$. 
Therefore, the contribution from the terms with $d=0$ in (\ref{sum 6.2}) is
\begin{eqnarray}
\label{MAIN++}
\frac{\pi^2}{4 D}\sum_{ \substack{ \a \in \mathbb{Z}^n  \\  1 \leq  j_1  \leq J  \\  0 \leq  j_r \leq \min\{J_r, j_1\} \\ (2\leq r\leq R)  } }
w_{\j}^* \left(  \frac{\a}{j_1}  \right)
\ll \left(\prod_{r=2}^R J_r\right)
\frac{1}{D} \sum_{1 \leq j_1 \leq J} j_1^{n} \ll \left(\prod_{r=2}^R J_r\right) \frac{  J^{n+1} }{D},
\end{eqnarray}
where the implicit constants depend only on $n$ and $L$. 

By the $n$-dimensional Poisson summation formula we obtain
\begin{eqnarray}
&&\sum_{  \a \in \mathbb{Z}^n  } w_{\j}^* \left(  \frac{\a}{j_1}  \right)   e\left( d j_1 F^*_{\j} \left(  \frac{\a}{j_1}  \right)  \right)
\\
&=&
\sum_{\k \in \mathbb{Z}^n} \int_{\mathbb{R}^n}   w_{\j}^* \left(  \frac{\z}{j_1}  \right)   e\left( d j_1 F^*_{\j} \left(  \frac{\z}{j_1}  \right) - \k \cdot \z  \right)       d\z
\notag
\\
\notag
&=&
j_1^{ n } \sum_{\k \in \mathbb{Z}^n} I_0(d; \j; \k),
\end{eqnarray}
where
$$
I_0(d; \j; \k) = \int_{\mathbb{R}^n}   w_{\j}^* (  \x ) e ( j_1 d  F^*_{\j} ( \x ) - j_1 \k \cdot \x   )    d\x.
$$
Therefore, in order to establish Proposition \ref{prop main}, it suffices to obtain a bound for
\begin{eqnarray}
&& \left| \sum_{ \substack{ 1 \leq  j_1  \leq J  \\  0 \leq  j_r \leq \min\{J_r, j_1\} \\ (2\leq r\leq R)   }   }
\sum_{1 \leq |d| \leq D} \frac{D - |d|}{D^2}
j_1^{n} \sum_{\k \in \mathbb{Z}^n} I_0(d; \j; \k) \right|
\\
\notag
&\leq&
2 \left|  \sum_{ \substack{  1 \leq  j_1  \leq J  \\  0 \leq  j_r \leq \min\{J_r, j_1\} \\ (2\leq r\leq R)    } }
\sum_{d=1}^D \frac{D - d}{D^2}
j_1^{n} \sum_{\k \in \mathbb{Z}^n} I_0(d; \j; \k) \right|,
\end{eqnarray}
where the inequality is obtained via complex conjugation.

Recall $\nabla F_{\j}$ is a diffeomorphism on $\overline{\mathcal{D}}$, $\nabla F_{\j} (\mathcal{D}) = \mathcal{R}_{\j}$, $\nabla F_{\j} (U_+) = V_{\j+}$
and $\nabla F_{\j}^* = (\nabla F_{\j})^{-1}$.
Let
\begin{equation}
\label{defrho'}
\rho' = \frac{1}{2} \textnormal{dist}(\partial \mathcal{D}, \partial U).
\end{equation}
As in Section \ref{easycase} we split the set of $\k \in \mathbb{Z}^n$ into three disjoint subsets.
We let
$$
\mathcal{K}_1 = \{ \k \in \ZZ^n : \frac{\k}{d} \in U \},
$$
$$
\mathcal{K}_2 = \{ \k \in \ZZ^n : \textnormal{dist} \left(\frac{\k}{d},  U \right) \geq \rho' \}
$$
and
$$
\mathcal{K}_3 = \ZZ^n \backslash  (\mathcal{K}_1 \cup \mathcal{K}_2).
$$
For each $1 \leq i \leq 3$, we define
$$
M_i =
 \sum_{d=1}^D \frac{D - d}{D^2} \left|
\sum_{ \substack{ 1 \leq  j_1  \leq J  \\  0 \leq  j_r \leq \min\{J_r, j_1\} \\  (2\leq r\leq R)   } }
j_1^{ n }    \sum_{\k \in \mathcal{K}_{i} } I_0(d; \j; \k)  \right|
$$
so that
\begin{eqnarray}
\label{MM1M2M3}
\mathcal{M} ( J, T^{-1}) \ll \left(\prod_{r=2}^R J_r\right) \frac{  J^{n+1} }{D}  +  M_1 + M_2 + M_3.
\end{eqnarray}
We now bound each $M_i$ separately.

\subsection{Case $\k \in \mathcal{K}_2$}
In this case we let
$$
\phi_1(\x) = \frac{ d  F^*_{\j} ( \x)  - \k\cdot \x}{\textnormal{dist} (\k, d U )}
$$
and
$$
\lambda_1 = j_1 \cdot \textnormal{dist} (\k, d U ).
$$
Then it follows that
$$
|\nabla \phi_1 (\x)|  = \frac{ |  d  \nabla F^*_{\j} ( \x)  - \k  | }{\textnormal{dist} (\k, d U )} \geq 1 \quad (\x \in V_{\j}),
$$
and furthermore by (\ref{Vj+}) that
$$
|\nabla \phi_1 (\x)| \geq \frac12 \quad (\x \in V_{\j+}).
$$

Next we need upper bounds for the derivatives of $\phi_1$.
\begin{lem}\label{Claim2}
Given $i_1, \ldots, i_n \in \ZZ_{\geq 0}$ with $\sum_{m=1}^n i_m\leq \ell$ we have
$$
\left| \frac{\partial^{i_1 + \cdots + i_n}  \phi_1}{ \partial x_1^{i_1} \cdots \partial x_n^{i_n} }  (\x)  \right| \ll 1 \quad (\x \in V_{\j+}),
$$
where the implicit constant is independent of $d, \j$ and $\k$ (In fact, it
depends only on $(i_1,\ldots, i_n)$,
upper bounds for (the absolute values of) finitely many derivatives of $f_r$ $(1\leq r\leq R)$ on $U_+$,
and the constants $c_1$ and $c_2$ in (\ref{lowerbdd'}).) Moreover, the same statement holds
with $\phi_1$ and $\ell$ replaced by $w_{\j}^*$ and $(\ell -1)$ respectively.
\end{lem}

\begin{proof}
We prove the statement for $F^*_{\j}$; the desired result for $\phi_1$ can then be proved in a similar manner as in the proof of Lemma \ref{Claim1}.
The statement for $w^*_{\j}$ also follows on recalling $w^*_{\j} = w \circ \nabla F_{\j}^*.$
We omit the latter details.

Since
$$
|\x \cdot \y| +|F_{\j}(\y)| \ll 1 \ (\x \in V_{\j+}, \y \in U_+),
$$
it follows from (\ref{dual}) that $|F^*_{\j}(\x)| \ll 1$ $(\x \in V_{\j+})$.

Next we recall that $\nabla F_{\j}^* = (\nabla F_{\j})^{-1}$.
Since $(\nabla F_{\j})^{-1} (V_{\j+}) = U_+$ we see that
$|\nabla F_{\j}^*(\x)| \ll 1$ $(\x \in V_{\j+})$.
By the chain rule we know that for $\x = \nabla F_{\j} (\y)$ with $\y \in U_+$, we have
\begin{equation}
\label{chain}
\textnormal{Jac}_{\nabla F^*_{\j}} (\x ) = \textnormal{Jac}_{(\nabla F_{\j})^{-1}} (\x ) =
(\textnormal{Jac}_{\nabla F_{\j}} (\y)  )^{-1};
\end{equation}
therefore, every second partial derivative of $F^*_{\j}$ is of the shape
\begin{equation}
\label{jaco1}
\frac{P}{ \det ( \textnormal{Jac}_{\nabla F_{\j}} (\y) ) },
\end{equation}
where $P$ is a degree $(n-1)$ real polynomial expression  (each coefficient is either $\pm 1$ or $0$)
in terms of the entries of  $\textnormal{Jac}_{\nabla F_{\j}} (\y)$.
It is clear that given any $i_1, \ldots, i_n \in \ZZ_{\geq 0}$ with $\sum_{m=1}^n i_m\leq \ell$ we have
$$
\left| \frac{\partial^{i_1 + \cdots + i_n}  F_{\j}}{ \partial y_1^{i_1} \cdots \partial y_n^{i_n} }  (\y)  \right| \ll 1 \quad (\y \in U_+).
$$
We also have $\textnormal{Jac}_{\nabla F_{\j}} = H_{F_{\j}}$. Therefore, from (\ref{lowerbdd'}) and (\ref{jaco1}) it follows that the absolute values of second partial derivatives of $F^*_{\j}$
are bounded by $\ll 1$.

Let $k \in \mathbb{N}$.
For higher partial derivatives, we note that by (\ref{chain}) and (\ref{jaco1})
any $k$-th partial derivative with respect to the $\x$-variables of an entry in $\textnormal{Jac}_{\nabla F^*_{\j}} (\x)$
is a degree $(n+kn+k)$ real  polynomial expression  
(the coefficients are independent of $\j$, but may depend on $k$)
in terms of:

1) (at most $(k+1)$-th) powers of $1/ \det ( \textnormal{Jac}_{\nabla F_{\j}} (\y) )$;


2) entries in $\textnormal{Jac}_{\nabla F_{\j}} (\y)$;

3) (at most $k$-th) partial derivatives with respect to the $\y$-variables of entries in $\textnormal{Jac}_{\nabla F_{\j}} (\y)$;

4) (at most $k$-th) partial derivatives with respect to the $\x$-variables of entries in $\nabla F^*_{\j} (\x) = (\nabla F_{\j})^{-1} (\x)$.
\newline
Therefore, the desired result follows from (\ref{lowerbdd'})  and induction.
\end{proof}

Therefore, it follows from Lemma \ref{Lem1} (with $\phi_1$ and $\lambda_1$) that
$$
I_0(d; \j; \k) \ll   \lambda_1^{- \ell + 1} =
(j_1 \cdot \textnormal{dist} (\k, d U ) )^{-\ell + 1},
$$
where the implicit constant is independent of $\j$ and $\k$.
As a result, since $\ell - 1 - n \geq 1$ it follows by a similar argument as in (\ref{(4.4)}) that
\begin{eqnarray}
\sum_{\k \in \mathcal{K}_2} I_0(d; \j; \k)
\ll  j_1^{-\ell + 1} \sum_{\k \in \mathcal{K}_2}  \textnormal{dist} (\k, d U )^{-\ell + 1}
\notag
\ll  j_1^{-\ell + 1},
\end{eqnarray}
where the implicit constant is independent of $d$.
Thus we obtain
\begin{eqnarray}
\label{M2bound}
M_2 &\leq& \sum_{ \substack{ 1 \leq  j_1  \leq J  \\ 0 \leq  j_r \leq \min\{J_r, j_1\} \\ (2\leq r\leq R) } }
\sum_{d=1}^D \frac{D - d}{D^2}
j_1^{ n }
\left| \sum_{\k \in \mathcal{K}_2}  I_0(d; \j; \k)  \right|
\\
\notag
&\ll&
\sum_{ \substack{ 1 \leq  j_1  \leq J  \\  0 \leq  j_r \leq \min\{J_r, j_1\} \\  (2\leq r\leq R)  } } j_1^{n  - \ell  + 1  }
\\
\notag
&\ll& \left(\prod_{r=2}^R J_r\right) \log J.
\end{eqnarray}

\subsection{Case $\k \in \mathcal{K}_3$}
\label{sec6.2}
Let $\lambda = j_1 d$ and
$$
\phi(\x) = F^*_{\j} (\x)  - \frac{\k}{d} \cdot \x.
$$
By our hypotheses, for each fixed $d$ we have that each $\k \in d \mathcal{D}$
determines a unique critical point $x_{d; \j; \k}$ of $\phi$, i.e.
$$
\mathbf{0} = \nabla \phi (\x_{d; \j; \k}) =
\nabla F^*_{\j} (\x_{d; \j;\k})  - \frac{\k}{d}.
$$
Then by (\ref{dual}) we have
\begin{eqnarray}
\label{(4.5)}
\phi(\x_{d; \j; \k}) = F^*_{\j}(\x_{d;\j; \k}) - \frac{\k}{d} \cdot \x_{d; \j; \k} = - F_{\j}\left( \frac{\k}{d} \right),
\end{eqnarray}
which we will make use of in the next section. It also follows from the definition of $\mathcal{R}_{\j}$ that
$$
\x_{d; \j; \k} = ( \nabla F^*_{\j} )^{-1} (\k / d) = \nabla F_{\j} (\k / d) \in \mathcal{R}_{\j}.
$$
\begin{lem}
\label{more lemma+}
We have
$$
\frac{| \x - \x_{d; \j;\k} |}{| \nabla \phi (\x) |} \ll 1 \quad (\x \in \nabla F_{\j} (\mathcal{D}_+), \x\neq \x_{d; \j;\k}),
$$
where the implicit constant is independent of $d$, $\j$ and $\k$.
\end{lem}
\begin{proof}
Since
$$
\frac{| \x - \x_{d; \j;\k} |}{| \nabla \phi (\x) |}
=
\frac{| \x - ( \nabla F_{\j}^* )^{-1} (\k / d) |}{| \nabla F_{\j}^* (\x) - \frac{\k}{d} |}
$$
and $\k/d \in \mathcal{D}$,
it suffices to prove
$$
\frac{| \x - \y |}{| \nabla F_{\j}^* (\x) - \nabla F_{\j}^* (\y) |} \ll 1 \quad (\x, \y \in \overline{\nabla F_{\j} (\mathcal{D}_+)}, \x\neq \y);
$$
this inequality is equivalent to
$$
\frac{| \nabla F_{\j} (\x') - \nabla F_{\j} (\y') |}{| \x' - \y' |} \ll 1 \quad (\x', \y' \in \overline{\mathcal{D}_+}, \x'\neq \y'),
$$
which is equivalent to the lower bound in (\ref{eqnlemmamore}).
\end{proof}
Since $(\nabla F_{\j})^{-1} (\x_{d;\j; \k}) = \k / d \not \in U$ we have $\x_{d; \j; \k} \not \in \nabla F_{\j} (U) = \supp w^*_{\j}$, i.e. $w^*_{\j}(\x_{d;\j; \k}) = 0$. The cardinality of $\mathcal{K}_{3}$ can be bounded as follows
\begin{eqnarray}
\label{cardK3}
\# \mathcal{K}_{3}  \leq
\sum_{  \substack{  \k \in \mathbb{Z}^n \\  \textnormal{dist} (\k/d, U ) < \rho' }} 1
 \ll_{U, \rho'} d^n.
\end{eqnarray}
Given $i_1, \ldots, i_n \in \ZZ_{\geq 0}$ with $\sum_{m=1}^n i_m\leq \ell$ we have
\begin{eqnarray}
\label{derivboundmainpropsec}
\left| \frac{\partial^{i_1} \cdots \partial^{i_n}  \phi}{ \partial x_1^{i_1} \cdots \partial x_n^{i_n} }  (\x) \right| \ll 1 \ \ (\x \in V_{\j+}),
\end{eqnarray}
where the implicit constant is independent of the choice of $d, \j$ and $\k$ in consideration;
this can be deduced from what we have shown in the proof of Lemma \ref{Claim2}.
We also have $H_{\phi} = H_{F_{\j}^*}$.
Therefore, it follows from Lemma \ref{Lem2}, (\ref{Hssnreln}) and (\ref{lowerbdd'}) that
\begin{eqnarray}
\sum_{\k \in \mathcal{K}_{3} }
I_0(d; \j; \k) \ll  d^n \lambda^{-\frac{n}{2} -1 } =   j_1^{-\frac{n}{2} -1} d^{ \frac{n}{2} - 1 }.
\end{eqnarray}
Consequently, we obtain
\begin{eqnarray}
\label{M3bound}
M_3 &=& \sum_{d=1}^D \frac{D - d}{D^2}
 \left| \sum_{ \substack{ 1 \leq  j_1  \leq J  \\  0 \leq  j_r \leq \min\{J_r, j_1\} \\(2\leq r\leq R)    } } j_1^{n} \sum_{\k \in \mathcal{K}_3}
  I_0(d; \j; \k) \right|
\\
&\ll&
\frac{1}{D} \sum_{ \substack{  1 \leq  j_1  \leq J  \\  0 \leq  j_r \leq \min\{J_r, j_1\} \\  (2\leq r\leq R)   } } j_1^{\frac{n}{2} - 1}
\sum_{d=1}^D d^{ \frac{n}{2} - 1 }
\notag
\\
\notag
&\ll& \left(\prod_{r=2}^R J_r\right) J^{\frac{n}{2} }   D^{ \frac{n}{2} - 1 }.
\end{eqnarray}

\subsection{Case $\k \in \mathcal{K}_1$}
Let $\lambda$ and $\phi$ be as in Section \ref{sec6.2}.
In particular, we have (\ref{(4.5)}) and Lemma \ref{more lemma+}. It can be verified that $\phi$ and its partial derivatives satisfy (\ref{derivboundmainpropsec}).
By (\ref{Hssnreln}) we have
$$
H_{ F_{\j}^*} (\x_{d ; \j; \k})  = H_{ F_{\j} } (\k/d)^{-1},
$$
and it follows that
\begin{equation}
\label{hessianeqn1}
\frac{1}{ \sqrt{|  \det  H_{ F_{\j}^* } (\x_{d ; \j; \k}) |} } =  \sqrt{ | \det  H_{ F_{\j} } (\k/d) | }.
\end{equation}
Since $\det (H_{f_1 + t_2 f_2 + \cdots + t_R f_R} (\y) ) \not = 0$ for all $\t \in [0,1]^{R-1}$ and $\y \in U$, and the eigenvalues of a matrix are continuous with respect to the coefficients of the matrix, it follows that the signature of $H_{f_1 + t_2 f_2 + \cdots + t_R f_R}(\y)$ does not change for $\t \in [0,1]^{R-1}$ and $\y \in U$. Thus we let $\sigma_0$ be the signaure of $H_{ \phi } (\x_{d ; \j; \k}) = H_{ F_{\j}^*} (\x_{d ; \j; \k})$,
which is constant for all $d, \j$ and $\k$ in consideration.
Therefore, it follows from Lemma \ref{Lem2}, (\ref{lowerbdd'}) and  (\ref{(4.5)}) that
\begin{eqnarray}
&&I_0(d; \j; \k)
\\
\notag
&=& \frac{w^*_{\j}(\x_{d; \j; \k})}{  \sqrt{ |  \det H_{ F^*_{\j} } (\x_{d;\j;\k})  | } } (j_1 d)^{-\frac{n}{2}}
e \left(  - j_1 d F_{\j} \left(  \frac{\k}{d}  \right)   +   \frac{\sigma_{\j}}{8} \right)
+  O\left( (j_1 d)^{ -\frac{n}{2} -1 } \right)
\\
\notag
&=& w \left( \frac{\k}{d} \right)  \sqrt{ | \det  H_{ F_{\j} } (\k/d) | }   (j_1 d)^{-\frac{n}{2}}
e \left(  - d ( j_1 f_1 + \cdots + j_R  f_R ) \left(  \frac{\k}{d}  \right)     +   \frac{\sigma_{\j}}{8} \right)
\\
\notag
&+&   O\left(  (j_1 d)^{ - \frac{n}{2}  - 1} \right),
\end{eqnarray}
where the implicit constant is independent of $d$, $\j$ and $\k$.

For $(u_1,\ldots, u_R)\in \R_{>0} \times \R_{\geq 0}^{R-1}$ we define the function
$$
\Psi_{\k; d}(u_1, u_2, \ldots, u_R) =
u_1^{\frac{n}{2}}  \sqrt{ | \det  H_{f_1 + \frac{u_2}{u_1}f_2 + \cdots + \frac{u_R}{u_1}f_R } ( \k / d ) |}.
$$
We note that $\Psi_{\k; d}( \, \cdot, u_2, \ldots, u_R) $ is a smooth function on
$\{u_1 \in \RR_{>0}: u_1 \geq u_r \,  (2 \leq r \leq R) \}$
for any fixed $u_2, \ldots, u_R \in \RR_{\geq 0}$.
Let
$$
\Psi_{\k; d}^{ (1) }  = \frac{\partial \Psi_{\k; d} }{ \partial u_1}.
$$
First we have
\begin{eqnarray}
\label{estimate I'1}
&& \left| \sum_{\substack{ 1 \leq j_1 \leq J  \\ 0 \leq  j_r \leq \min\{J_r, j_1\} \\ (2\leq r\leq R)    }} j_1^{n} I_0(d; \j; \k) \right|
\\
&\ll&
\notag
w \left( \frac{\k}{d} \right) d^{\frac{-n}{2}}
\sum_{ \substack{  0 \leq  j_r \leq  J_r \\ (2\leq r\leq R)}}
\left| \sum_{ \max\{1, j_2, \ldots, j_R  \} \leq j_1 \leq J }  \Psi_{\k; d}(j_1, \ldots,  j_R) e \left(  - d j_1 f_1( \k / d) \right) \right|
\\
\notag
&+&  O \left( \left(\prod_{r=2}^R J_r\right) J^{\frac{n}{2} } d^{- \frac{n}{2} - 1} \right).
\end{eqnarray}
We need the following lemma to estimate the above sum.
\begin{lem}
\label{bound''}
Let $u_1, \ldots, u_R \in \RR$ be such that $u_1 > 0$ and $0 \leq u_2, \ldots, u_R  \leq u_1.$
Then for any $\k \in \mathcal{K}_1$ we have 
\begin{eqnarray}
\notag
| \Psi_{\k; d} (u_1, \ldots,  u_R) | \ll u_1^{\frac{n}{2}}
\end{eqnarray}
and
\begin{eqnarray}
\notag
| \Psi_{\k; d}^{(1)} (u_1, \ldots,  u_R) | \ll u_1^{\frac{n}{2} - 1},
\end{eqnarray}
where the implicit constants are independent of $d$ and $\k$.
\end{lem}
\begin{proof}
The first estimate follows trivially from (\ref{lowerbdd'}).
Let
$$
\det H_{f_1 + x_2 f_2 + \cdots + x_R f_R } (\k/d)  = \sum_{  \substack{ 0 \leq  \ell_2 +  \cdots +  \ell_R \leq  n  \\  \ell_2, \ldots, \ell_R \geq 0 }  } A_{\ell_2, \ldots, \ell_R} x_2^{\ell_2} \cdots x_R^{\ell_R}.
$$
Since $\k/d \in U$ it follows that $| A_{\ell_2, \ldots, \ell_R} | \ll 1,$
where the implicit constant is independent of $d$ and $\k$.
For $u_1 > 0$ and $0 \leq u_2, \ldots, u_R  \leq u_1$, we have
\begin{eqnarray}
\notag
&&| \Psi_{\k; d}^{(1)} (u_1, \ldots,  u_R) |
\\
\notag
&\ll&
\frac{n}{2}  u_1^{\frac{n}{2} - 1 } \sqrt{ | \det  H_{ f_1 + \frac{u_2}{u_1}f_2 + \cdots + \frac{u_R}{u_1}f_R }(\k/d)|  }
\\
\notag
&+&
\frac{  u_1^{\frac{n}{2}  } } { 2 \sqrt{ | \det  H_{ f_1 + \frac{u_2}{u_1}f_2 + \cdots + \frac{u_R}{u_1}f_R  }(\k/d)|  } }  \sum_{  \substack{ 0 \leq  \ell_2 +  \cdots +  \ell_R \leq  n  \\  \ell_2, \ldots, \ell_R \geq 0 }  } |A_{\ell_2, \ldots, \ell_R}| \frac{u_2^{\ell_2} \cdots u_R^{\ell_R}}{u_1^{\ell_2 + \cdots + \ell_R + 1}}
\\
\notag
&\ll&
u_1^{\frac{n}{2} - 1}.
\end{eqnarray}
\end{proof}
With these estimates in hand, it follows by partial summation
\begin{eqnarray}
\notag
\left| \sum_{ \max\{1, j_2, \ldots, j_R  \} \leq j_1 \leq J }  \Psi_{\k; d}(j_1, \ldots,  j_R) e \left(  - d j_1 f_1( \k / d) \right) \right|
\ll  J^{ \frac{n}{2} }  \min \{ J,  \frac{1}{ \|  d  f_1( \k /d)   \| } \}.
\end{eqnarray}
Therefore, we obtain
\begin{eqnarray}\label{sumsumsum3}
M_1 &\ll& \frac{1}{D} \sum_{d=1}^D  \sum_{\k \in \mathcal{K}_1}   w \left( \frac{\k}{d} \right) d^{-\frac{n}{2}}
\left(\prod_{r=2}^R J_r\right)J^{ \frac{n}{2}  }  \min \{ J,  \frac{1}{ \|  d  f_1( \k /d)   \| } \}
\\
\notag
&+&
\frac{1}{D}  \sum_{d=1}^D  \sum_{\k \in \mathcal{K}_1} \left(\prod_{r=2}^R J_r\right) J^{\frac{n}{2} } d^{-\frac{n}{2} - 1}.
\end{eqnarray}
By a similar argument as in (\ref{cardK3}) we have
$$
\# \mathcal{K}_{1} \ll d^n.
$$
Thus the second term in (\ref{sumsumsum3}) can be bounded by
\begin{eqnarray}
\label{sumsumsum3'}
\left(\prod_{r=2}^R J_r\right) \frac{ J^{ \frac{n}{2}  } }{D}
  \sum_{d = 1 }^D \sum_{\substack{ \k \in \mathcal{K}_{1}  } }  d^{-\frac{n}{2} - 1}
  &\ll&
\left(\prod_{r=2}^R J_r\right)\frac{J^{\frac{n}{2} }   }{D}
\sum_{d = 1 }^D
d^{\frac{n}{2} - 1}
\\
\notag
&\ll& \left(\prod_{r=2}^R J_r\right) J^{\frac{n}{2}  } D^{\frac{n}{2} - 1}.
\end{eqnarray}
In order to estimate the first sum in (\ref{sumsumsum3}), we use the following result \cite[Theorem 2]{JJH}: For any $X > 0$ we have
\begin{eqnarray}
\sum_{d=1}^D
\sum_{\substack{ \k \in \ZZ^n \\ \| d f_1 \left(  \frac{\k}{d} \right)  \|  \leq  X^{-1} } }
w \left( \frac{\k}{d} \right)   \ll   X^{-1}D^{n+1} +  D^n \mathcal{E}_n(D),
\end{eqnarray}
where
$$
\mathcal{E}_n(D)
=
\begin{cases}
    \exp( \mathfrak{c}_3 \sqrt{\log D}) & \mbox{if } n = 2 \\
    (\log D)^{\mathfrak{c}_4} & \mbox{if } n \geq 3 \\
\end{cases}
$$
for some positive constants $\mathfrak{c}_3$ and $\mathfrak{c}_4.$ Here the constants $\mathfrak{c}_3$ and $\mathfrak{c}_4$ and the implicit constants  depend only on $n$, $c_1$ and $c_2$ in (\ref{lowerbdd'}), $\rho$ in (\ref{defnrho}), $\rho'$ in (\ref{defrho'}), and upper bounds for (the absolute values of) finitely many derivatives of $w$ and $f_1$ on $\mathcal{D}_+$. \par
By partial summation we have
\begin{eqnarray}
\label{partial sum++}
\sum_{ 1 \leq  d \leq D  }
\sum_{\substack{  \substack{  \k \in \ZZ^n  \\  \| d f_1 \left(  \frac{\k}{d} \right)  \|  \leq  X^{-1}   }   } } w \left( \frac{\k}{d} \right) d^{- \frac{ n}{2} }
\ll
D^{-\frac{n}{2}}  (X^{-1} D^{n+1} + D^n \mathcal{E}_n(D)).
\end{eqnarray}
Therefore, it follows that
\begin{eqnarray}
\label{partial sum+++}
&& \left(\prod_{r=2}^R J_r\right) \frac{J^{ \frac{n}{2}  }  }{D}
\sum_{1 \leq d \leq D}  \sum_{\substack{ \k \in \ZZ^n  } }  w \left( \frac{\k}{d} \right) d^{-\frac{n}{2}}  \min \{ J, \frac{1}{\| d f_1 \left(  \frac{\k}{d} \right)  \|} \}
\\
\notag
&\ll & \left(\prod_{r=2}^R J_r\right)  \frac{J^{ \frac{n}{2} + 1 }  }{D}
D^{-\frac{n}{2}}  (J^{-1} D^{n+1} + D^n \mathcal{E}_n(D))
\\
\notag
&+&
\left(\prod_{r=2}^R J_r\right)  \frac{J^{ \frac{n}{2}  }  }{D}
\sum_{1 \leq d \leq D}  \sum_{\substack{ \k \in \ZZ^n \\   J^{-1} <  \| d f_1 \left(  \frac{\k}{d} \right)  \| } }
w \left( \frac{\k}{d} \right) d^{-\frac{n}{2}}   \frac{1}{\| d f_1 \left(  \frac{\k}{d} \right)  \|}. 
\end{eqnarray}
In order to estimate the final term in (\ref{partial sum+++}), we split the interval $[J^{-1}, 1/2]$ into dyadic intervals.
Since the sum is an empty sum if $J^{-1}> 1/2$, we assume $J^{-1} \leq 1/2$.
Then by (\ref{partial sum++}) we have
\begin{eqnarray}
\label{partial sum++++}
&& \left(\prod_{r=2}^R J_r\right) \frac{J^{ \frac{n}{2}  }  }{D} \sum_{ 1 \leq d \leq D}  \sum_{ \substack{ \k \in  \ZZ^n  \\ J^{-1}  <  \|  d f_1(\k/d)  \|  } }
w \left(  \frac{\k}{d}  \right) d^{-\frac{n}{2}}
\left\|  d f_1 \left(  \frac{\k}{d}  \right)  \right\|^{-1}
\\
\notag
&\leq& \left(\prod_{r=2}^R J_r\right) \frac{J^{ \frac{n}{2}  }  }{D} \sum_{1 \leq i \leq \frac{\log J}{\log 2} + 1 } J 2^{1 - i}
\sum_{ 1 \leq d \leq D} d^{-\frac{n}{2}}
\sum_{ \substack{   \k \in \ZZ^n  \\
\frac{2^{i-1}}{J}  <  \|  d f_1(\k/d)  \| \leq   \frac{2^{i}}{J} } }
w \left(  \frac{\k}{d}  \right)
\\
\notag
&\ll & \left(\prod_{r=2}^R J_r\right) \frac{J^{ \frac{n}{2}  }  }{D} \sum_{1 \leq i \leq \frac{\log J}{\log 2} +  1} J 2^{1 - i}
 D^{-\frac{n}{2} }   (2^i J^{-1} D^{n+1} + D^n \mathcal{E}_n(D))
\\
&\ll&
\notag
\left(\prod_{r=2}^R J_r\right) \frac{J^{ \frac{n}{2}  }  }{D} \left( (\log J) D^{\frac{n}{2} + 1 } + J D^{\frac{n}{2} }  \mathcal{E}_n(D) \right).
\end{eqnarray}
Therefore, by combining (\ref{sumsumsum3}), (\ref{sumsumsum3'}), (\ref{partial sum+++}) and (\ref{partial sum++++}) we obtain
\begin{eqnarray}
\label{M1lastbdd}
M_1& \ll& \left(\prod_{r=2}^R J_r\right)\frac{J^{  \frac{n}{2} + 1} }{D}  D^{-\frac{n}{2}}  (J^{-1} D^{n+1} + D^n \mathcal{E}_n(D))
 \\
\notag
&+&  \left(\prod_{r=2}^R J_r\right) \frac{J^{  \frac{n}{2} } }{D}  \left( (\log J) D^{\frac{n}{2} + 1 } + J D^{\frac{n}{2} }  \mathcal{E}_n(D) \right) +  \left(\prod_{r=2}^R J_r\right) J^{\frac{n}{2}  } D^{\frac{n}{2} - 1}
\\
&\ll&
\notag
\left(\prod_{r=2}^R J_r\right) \left(
(\log J)J^{ \frac{n}{2}  } D^{ \frac{n}{2}} 
+ J^{ \frac{n}{2} + 1 } D^{ \frac{ n}{2} - 1} \mathcal{E}_n(D) \right).
\end{eqnarray}

\subsection{Final estimate}
Recall $D = \lfloor T/2 \rfloor$ and $T \geq 2$.
Combining  (\ref{MAIN++}), (\ref{M2bound}),  (\ref{M3bound}) and (\ref{M1lastbdd}) yields
\begin{eqnarray}
\notag
\mathcal{M}(J, T^{-1} ) 
&\ll & \left(\prod_{r=2}^R J_r\right)\frac{J^{n + 1}}{T}   
+ \left(\prod_{r=2}^R J_r\right)(\log J)J^{ \frac{n}{2}  } T^{ \frac{n}{2}}
\notag
\\
&+&  \left(\prod_{r=2}^R J_r\right)J^{ \frac{n}{2} + 1 } T^{ \frac{n}{2} - 1} \mathcal{E}_n(T).
\notag
\end{eqnarray}

Suppose $T^{-1} \leq J^{-1}$. Then we have
\begin{eqnarray}
\mathcal{M}(J, T^{-1} ) \leq
\mathcal{M}(J, J^{-1} )
\label{mainpropend1}
\ll
\left(\prod_{r=2}^R J_r\right) J^n (\mathcal{E}_n(J) + \log J).
\end{eqnarray}
On the other hand, if $T^{-1} > J^{-1}$, i.e. $J > T$, then
\begin{eqnarray}
\label{mainpropend2}
\mathcal{M}(J, T^{-1} )  \ll  \left(\prod_{r=2}^R J_r\right)\frac{J^{n + 1}}{T} +
\left(\prod_{r=2}^R J_r\right) J^n (\mathcal{E}_n(J) + \log J).
\end{eqnarray}
Finally, Proposition \ref{prop main} follows from (\ref{mainpropend1}) and (\ref{mainpropend2}).

\begin{rem} \label{REM1}
The positive constants $\mathfrak{c}_1$ and $\mathfrak{c}_2$ and the implicit constants in
the statement of Theorem \ref{main upper bound thm} depend only on
$n$, $R$, $c_1$ and $c_2$ in (\ref{lowerbdd'}), $\rho$ in (\ref{defnrho}), $\rho'$ in (\ref{defrho'})
for each $(r; \boldsymbol{\epsilon})$ $(1 \leq r \leq R, \boldsymbol{\epsilon} \in \{\pm 1\}^R)$
(Recall the remark made after (\ref{N^1}) and that the references given here for $c_1, c_2, \rho$ and
$\rho'$ are for the case $(r; \boldsymbol{\epsilon}) = (1; (1, \ldots, 1))$.) and upper bounds for (the absolute values of) finitely many derivatives of $w$ and $f_r$ $(2\leq r\leq R)$ on $B_{3 \tau /2} (\x_0)$. Also we may replace the assumption that $f_r \in C^{\ell}(\RR^n)$ with $f_r \in C^{\ell}(B_{\eta}(\x_0))$ $(1 \leq r \leq R)$
for any $\eta > 0$. 
\end{rem}

\section{Examples}\label{examples}

Our first example is based on the construction of certain matrices by A. A. Suslin in his work \cite{Sus} on
stably free modules. For a similar construction we refer the reader to see \cite[Section 14.2.3]{HKMS}. 

Example 1: Let $A_2 (t_1, t_2) = \begin{pmatrix}
       t_2 & t_1 \\
       t_1 & - t_2
     \end{pmatrix}$, and for $R \geq 3$ we let
$$
A_R(t_1, \ldots, t_R) = \begin{pmatrix}
       t_R I_{2^{R-2}}  & A_{R-1}(t_1, \ldots, t_{R-1}) \\
       A_{R-1}(t_1, \ldots, t_{R-1}) & - t_R I_{2^{R-2}}
     \end{pmatrix},
$$
where $I_{m}$ denotes the $m \times m$ identity matrix.
Then
$$
A_R(t_1, \ldots, t_R) = t_1 A_R(1, 0, \ldots, 0) + t_2 A_R(0, 1, 0 \ldots, 0) + \cdots + t_R A_R(0,  \ldots, 0, 1),
$$
and $A_R(1, 0, \ldots, 0),$ $A_R(0, 1, 0 \ldots, 0),$ $\ldots, A_R(0,  \ldots, 0, 1)$ are real symmetric matrices.
\\

Claim: For each $R \geq 2$, $\det A_R(t_1, \ldots, t_R) \neq 0 $
for all $(t_1, \ldots, t_R) \in \RR^R \backslash \{ \mathbf{0} \}$.
\begin{proof}
We prove by induction that
$$
A_{R}(t_1, \ldots, t_{R})^2 = (t_1^2 + \cdots + t_{R}^2) I_{2^{R-1}}.
$$
Then it follows that
$$
(\det A_{R}(t_1, \ldots, t_{R}))^2 = (t_1^2 + \cdots + t_{R}^2)^{2^{R-1}},
$$
and the result is immediate.
For the base case $R=2$ we have
$$
A_2(t_1, t_2)^2 = \begin{pmatrix}
       t_2 & t_1 \\
       t_1 & - t_2
     \end{pmatrix}^2
     =
     \begin{pmatrix}
       t_1^2 + t_2^2 & 0 \\
       0 & t_1^2 +  t_2^2
     \end{pmatrix}.
$$
Suppose the statement holds for some $R \geq 2$.
Then it follows that
\begin{eqnarray}
&&A_{R+1}(t_1, \ldots, t_{R+1})^2
\notag
\\
&=&  \begin{pmatrix} t_{R+1}^{2} I_{2^{R-1}} + A_R(t_1, \ldots, t_R)^2 & 0 \\  0 &  t_{R+1}^{2} I_{2^{R-1}} + A_R(t_1, \ldots, t_R)^2 \end{pmatrix}
\notag
\\
\notag
&=& (t_1^2 + \cdots + t_{R+1}^2) I_{2^{R}}.
\end{eqnarray}

\end{proof}

The next construction stems from the field of \textit{determinantal representation} (see for example \cite{GKKP, NT, Q}), where given a polynomial $g(x_1, \ldots, x_R) \in \RR[x_1, \ldots, x_R]$ one seeks  to find
$n \times n$ real matrices $H_0, \ldots, H_R$ such that
$$
g(x_1, \ldots, x_R) = \det \left( H_0 + \sum_{i=1}^R x_i H_i \right).
$$
Even though we have an additional restriction that $H_0$ is the zero matrix, we can nevertheless
make use of the techniques developed in this area to find further examples.

Example 2: Let $R \geq 2$ and $n = 2^{\lfloor R / 2 \rfloor - 1}$.
It follows from \cite[Theorem 5.3]{NT} (note the assumption on \cite[pp.1580]{NT}) that
there exist $2n \times 2n$ Hermitian matrices $M_1, \ldots, M_R$ such that
$$
\det( I_{2n} + x_1 M_1 + \cdots + x_R M_R) = (1 - (x_1^2 + \cdots  + x_R^2 ))^n.
$$
We refer the reader to \cite[Example 4.5]{NT}
for an explicit example of such $M_1, \ldots, M_R$. Let $\epsilon > 0$.
Then by replacing $x_i$ with $\frac{x_i}{\epsilon}$, and multiplying both sides by $\epsilon^{2n}$, the above equation becomes
$$
\det ( \epsilon I_{2n} +  x_1 M_1 + \cdots + x_R M_R) = ( \epsilon^2 - (x_1^2 + \cdots + x_R^2 ))^n.
$$
Noting that both sides are polynomials in $\epsilon, x_1, \ldots, x_R$, by taking the limit $\epsilon \to 0$ we obtain  $\det( x_1 M_1 + \cdots + x_R M_R) =(-1)^n (x_1^2 + \cdots  + x_R^2 )^n$. 
Finally, by \cite[Lemma 2.14]{NT} we obtain $4n \times 4n$ real symmetric matrices $H_1, \ldots, H_R$ such that
$$
\det( x_1 H_1 + \cdots + x_R H_R) = (x_1^2 + \cdots  + x_R^2 )^{2n}.
$$

\end{document}